%% file: lagr-cm.tex
\documentclass[a4paper,11pt,reqno]{article}

\usepackage[utf8]{inputenc}
\usepackage[T1]{fontenc}
\usepackage{lmodern}
\usepackage[english]{babel}
\usepackage{csquotes}
\usepackage{microtype}
\usepackage{graphicx,color}

\usepackage{soul}

\usepackage[a4paper, hmargin=3.25cm, top=3.5cm, bottom=4cm,
	footskip=3\baselineskip]{geometry}

\usepackage{tikz}
\definecolor{midnightblue}{rgb}{0.1, 0.1, 0.44}
\usepackage[colorlinks,allcolors=midnightblue]{hyperref}
\newcommand*{\mailurl}[1]{\href{mailto:#1}{\nolinkurl{#1}}}
\AtBeginDocument{}

\usepackage{fsmath}

\newcommand{\Ball}{\mathrm{B}} 
\newcommand{\Rsp}{\mathbb{R}}
\newcommand{\Zsp}{\mathbb{Z}} 
\newcommand{\E}{E} 
\newcommand{\Card}{\mathrm{Card}} 
\newcommand{\Ent}{\mathcal{H}} 
\newcommand{\Probac}{\Prob_\text{ac}} 
\newcommand{\Wass}{W}

\newcommand{\dd}{\d}
\newcommand{\nr}{\norm}
\newcommand{\sca}[2]{\langle#1|#2\rangle}

\title{Lagrangian discretization\\ of crowd motion and linear diffusion}
\author{
	Hugo Leclerc\thanks{Laboratoire de Math\'ematiques d'Orsay,
	Université Paris-Sud,
	\mailurl{hugo.leclerc@math.u-psud.fr}.}
\and
	Quentin Mérigot\thanks{Laboratoire de Math\'ematiques d'Orsay,
	Université Paris-Sud,
	\mailurl{quentin.merigot@math.u-psud.fr}.}
\and
	Filippo Santambrogio\thanks{Institut Camille Jordan,
	Université Claude Bernard Lyon 1, 
	\mailurl{santambrogio@math.univ-lyon1.fr}.}
\and
	Federico Stra\thanks{Chaire d'Analyse Math\'ematique, Calcul des Variations et EDP, EPFL, 
	\mailurl{federico.stra@epfl.ch}.}
}

\begin{document}

\maketitle

\begin{abstract}
We study a model of crowd motion following a gradient vector field,
with possibly additional interaction terms such as
attraction/repulsion, and we present a numerical scheme for its
solution through a Lagrangian discretization. The density constraint
of the resulting particles is enforced by means of a partial optimal
transport problem at each time step.  We prove the convergence of the
discrete measures to a solution of the continuous PDE describing the
crowd motion in dimension one. In a second part, we show how a similar
approach can be used to construct a Lagrangian discretization of a
linear advection-diffusion equation. Both discretizations rely on the
interpretation of the two equations (crowd motion and linear
diffusion) as  gradient flows in Wasserstein space.  We provide also a
numerical implementation in 2D to demonstrate the feasibility of the
computations.
\end{abstract}


\input{sec1-intro}
\input{sec2-crowd}
\input{sec3-diffusion}
\input{sec4-pressure}
\input{sec5-numerics}

\paragraph*{Acknowledgements}
This work has been supported by Agence nationale de la recherche
(ANR-16-CE40-0014 - MAGA - Monge-Amp\`ere et G\'eom\'etrie
Algorithmique). The authors would like to thank the anonymous referees
for several remarks which lead to an improvement of the manuscript.

\bibliographystyle{plain}
\bibliography{biblio}

\end{document}

%% file: sec1-intro.tex
\section{Introduction}

In this paper we present an approximation scheme to solve evolution
PDEs which have a gradient-flow structure in the space of probability
measures $\mathcal{P}(\Omega)$ endowed with the Wasserstein distance
$W_2$. Here, $\Omega\subset\setR^d$ is a given compact domain where
the evolution takes place, and the PDE will naturally be complemented
by no-flux boundary conditions. The approximation that we present is
Lagrangian in the sense that the evolving measure $\rho_t$ will be
approximated by an empirical measure of the form $\frac
1N\sum_{i=1}^N\delta_{x_i(t)}$ and we will look for the evolution of
the points $x_i$. We will use this approximation to provide an
efficient numerical method, based on the most recent developments in
semi-discrete optimal transport
\cite{de2012blue,levy2015numerical,kitagawa2016convergence,bourne2018semi}. Here,
``semi-discrete'' refers to the fact that the discretization of the
diffusion effects in the evolution equation involves computation of
the optimal transport plans between an empirical measure and diffuse
measures.

Starting from the pioneering work of Otto and Jordan-Kinderlherer-Otto
\cite{otto2001geometry,jordan1998variational} it is well-known that
some linear and non-linear diffusion equations can be expressed in terms of a
gradient flow in the space $W_2(\Omega)$. More precisely, the
Fokker-Planck equation
$$\partial_t\rho-\Delta\rho-\div(\rho\nabla V)=0$$ is the
gradient flow of the energy $E(\rho):= \int \rho\log\rho+\int Vd\rho$,
and the porous-medium equation
$$\partial_t\rho-\Delta\rho^m-\div(\rho\nabla V)=0, m\in (1,+\infty)$$ is the
gradient flow of the energy $E(\rho):=\int \frac{1}{m-1}\rho^m+\int
Vd\rho$, Recently, also the limit case $m=\infty$ has been considered,
in the framework of crowd motion \cite{gf-cm}. In this case, the
functional is $E(\rho):=\int Vd\rho$ if $\rho$ satisfies the
constraint $\rho\leq 1$, and $E(\rho)= +\infty$ otherwise; the
corresponding PDE is
$$
\begin{cases}
  \partial_t \rho + \div(\rho v) = 0  \\
  v =  -\nabla p - \nabla V \\
  0\leq \rho \leq 1, ~ 
  p\geq 0, ~
  p(1-\rho) = 0.
\end{cases}
$$ One can see the appearance of a pressure $p$ accounting for the
constraint $\rho\leq 1$.  We will come back later to the precise
meaning and formulations of this last equation.

\paragraph{Approximation by empirical measures}
Since any probability measure can be approximated by empirical
measures, it is tempting to perform an approximation scheme just by
considering the gradient flow of one of the above energy functionals
$E$ on the set $\mathcal{P}_N(\Omega)$ of uniform measures on $N$
atoms, and then let $N\to\infty$. Unfortunately, the domain of the
above functionals is reduced to absolutely continuous measures, and
its intersection with $\mathcal{P}_N(\Omega)$ is empty. The main idea
and novelty of this paper is to write $E(\rho) = F(\rho) + \int
V\dd\rho$ where $F$ is the entropy or the congestion constraint,
$$\begin{aligned}
  &F(\rho) = \int \rho\log\rho &\hbox{(linear diffusion)} \\
  \hbox{ or } & F(\rho)
= \begin{cases} 0 & \hbox{ if } \rho \leq 1 \\ +\infty & \hbox{ if
    not} \end{cases}  &\hbox{(crowd motion, m=$+\infty$),}
\end{aligned}$$
and to replace $F$ by its Moreau-Yosida regularization
$$F_\varepsilon(\mu):=\inf_\rho
F(\rho)+\frac{1}{2\varepsilon}W_2^2(\mu,\rho).$$ The energies
$F_\varepsilon$ are finite and well-defined for arbitrary probability
measures $\mu$, and converge to $F$ as $\varepsilon\to 0$. More
importantly, we will see that it is possible to compute very
efficiently $F_\eps(\mu)$ when $\mu\in \mathcal{P}_N(\Omega)$.  The
evolution of the discrete measures is then dealt with by keeping track
of the positions of the particles $X=(x_1,\dotsc,x_N)\in\setR^{Nd}$ in
the support of the associated measure $\mu_X = \frac1N \sum_{i=1}^N
\delta_{x_i}$.  Thanks to the correspondence between $X$ and $\mu_X$,
we can think of $F_\eps$ as an energy on the space of particle
positions too, given by
\begin{equation}\label{eq:D}
F_\eps(x_1,\hdots,x_N) = F_\eps\left(\frac{1}{N} \sum_{1\leq i\leq
  N} \delta_{x_i}\right)
\end{equation}
The discrete gradient flow then takes the form of a system of ODEs
\begin{equation}
  \begin{cases}
    \frac{1}{N} \dot{x}_i(t)  = - \nabla_{x_i} F_{\eps_N}(x_1(t),\hdots,x_N(t)) - \frac{1}{N}\nabla V(x_i(t)), \\
    X^N(0) = X^N_0,
  \end{cases}
\end{equation}
for a suitable choice of $\eps_N\to0$. The particles are only coupled
by the forces $-\nabla_{x_i} F_{\eps_N}(x_1,\hdots,x_N)$ due to
diffusion or to the congestion constraint. Finally, we note that the
initial condition can be selected by optimal quantization of the
initial density $\rho_0\in\Prob(\Omega)$,
\begin{equation}\label{eq:initial}
X^N_0 \in \argmin_{X\in\setR^{Nd}} W_2^2(\rho_0, \mu_X),
\end{equation}
granting, in many situations, an initial error of approximately $W_2^2(\rho_0, \mu_{X_0^N}) \lesssim (1/N)^{1/d}$.

\paragraph{Convergence}
In the paper we will present the approximation scheme and prove, in
these two cases, the convergence of the curves of empirical measures
to the solution of the corresponding PDE, under the assumption that a
certain bound on the approximate solutions themselves is
satisfied. This assumption is unnatural, and it would be desirable to
remove it, or replace it with an assumption on the approximation of
the initial data. Yet, this seems to be a non-trivial problem, which
is closely related to the general question of the convergence of
gradient flows once the functionals $\Gamma$-converge. We refer to
\cite{SanSer,serfaty2011gamma} as classical papers on this
question. In these papers, a semi-continuity property on the slopes of
the functionals is required, which is in the same spirit of the bounds
we need. These required bounds are stronger in the crowd motion case,
as the equation is non-linear and stronger compactness is needed,
while sligthly weaker in the Fokker-Planck case, which is indeed a
linear equation.  We show that such bounds can be obtained in
dimension 1 (see \S\ref{sec:bounds}).  However, we insist that these
bounds can be verified numerically in general, which makes the scheme
we propose interesting for the approximation in arbitrary dimension.

\paragraph{Comparison to existing Lagrangian schemes}
The optimal transport interpretation of advection-diffusion equations
by Otto and Jordan-Kinderlherer-Otto
\cite{otto2001geometry,jordan1998variational} has already led to many
Lagrangian schemes:

\begin{itemize}
\item In dimension $d=1$ it is quite easy to construct such Lagrangian
  schemes. This is due to the fact that the Wasserstein space
  $(\Prob(\Rsp),\Wass_2)$ can be isometrically embedded into
  $L^2([0,1])$ through the inverse cumulative distribution
  function. Discretizing probability densities in a lagrangian way
  then amounts to discretizing inverse cdfs, which can be done using
  finite elements of order one \cite{blanchet2008convergence} or two
  \cite{matthes2017convergent}, leading respectively to piecewise
  constant or piecewise linear densities.
\item In dimension $d\geq 2$, one cannot isometrically embed the
  Wasserstein space in a $L^p$ space, making the choice of
  discretization less canonical. There exists discretization based on
  piecewise constant densities over Voronoi cells
  \cite{carrillo2017numerical}, or using Gaussian mixtures (``blobs'')
  \cite{carrillo2019blob}. Evans, Savin and Gangbo
  \cite{evans2005diffeomorphisms} have proposed a general way to
  rewrite some Wasserstein gradient flows as gradient flows in the
  space of diffeomorphisms (the idea is to write $\rho_t = s_{t\#}
  \rho_0$ where $s_t$ is a time-dependent diffeomorphism), which is
  used to construct a Lagrangian discretization in
  \cite{junge2017fully}.
\end{itemize}

Our scheme could be considered as a variant of the scheme of
\cite{carrillo2017numerical} involving Voronoi cells but it is, to the
best of our knowledge, the first one to use Laguerre cells (see
\S\ref{sec:comput-my}), which are objects intrinsically related to the
Wasserstein structure of the equation. The use of Laguerre cells, easy
to handle via modern computational geometry tools, is now
state-of-the-art when the transport cost is quadratic. However, we
stress that we only use quadratic transport costs in the computation
of the Moreau-Yosida regularization of the diffusion or congestion
term, which means that we could a priori think of attacking with the
same ideas other PDEs, which have a gradient structure for other
distances. We refer for instance to \cite{agueh2005existence} for PDEs
induced by a cost functions of the form $c(x,y) = |x-y|^q$ or
\cite{mccann2009constructing} for the so-called relativistic heat
equation.

%% file: sec2-crowd.tex
\section{Lagrangian discretization of crowd motion}
\label{sec:crowd-motion}

\paragraph{Formulation of the continuous problem}
We fix a compact domain $\Omega\subset\setR^d$ and a potential $V\in
C^1(\setR^d)$ bounded from below, e.g. $V\geq 0$.  The crowd is
described by a probability measure $\rho$ in $\Omega$. Each agent
tries to follow the gradient vector field $-\nabla V$ while ensuring
that the probability density satisfies the density constraint
$\rho\leq1$.  Therefore, we introduce the constraint set
\begin{equation} \label{eq:K}
K = \set{\rho\in\Probac(\Omega)}{\rho \leq 1}, 
\end{equation}
and we assume that $\abs{\Omega}\geq 1$ so that $K$ is non-empty.
There are a few possible ways to express this idea of constrained
motion, which are at least formally equivalent.  The first version is
straightforward, but a bit problematic to formulate rigorously. The
mass evolution is expressed by a continuity equation where the driving
vector field is the projection of $-\nabla V$ onto the tangent cone of
the set $K$:
\begin{equation}\label{eq:cm-pde-proj}
  \begin{cases}
    \partial_t \rho + \div(\rho v) = 0, & \rho\in K, \\
    v = \Pi_{T_\rho K}(-\nabla V),
  \end{cases} \\
\end{equation}
where
\[
T_\rho K = \set*{v\in L^2(\rho;\setR^d)}{
	\div(v) \geq 0 \text{ on } \{\rho=1\},
	\ v\cdot n \leq 0 \text{ on } \partial\Omega\cap\{\rho=1\} }.
\]
Note that the choice of the boundary conditions on $v\cdot n$ is not relevant at this stage, since anyway the equation $ \partial_t \rho + \div(\rho v) = 0$ is already intended with no-flux boundary conditions $\rho v\cdot n=0 $, i.e. $v\cdot n=0$ on $\{\rho>0\}$. Choosing to impose $v\cdot n\leq 0$ on the part of the boundary where the density $\rho$ is saturated allows for an easier expression of the normal cone below.
%
Indeed, projecting to the tangent cone amounts to subtracting a vector from
the normal cone, dual to the tangent cone, thereby leading to
\begin{equation}\label{eq:cm-pde-norm}
\begin{cases}
\partial_t \rho + \div(\rho v) = 0, & \rho\in K, \\
v = -\nabla V - w, & w\in N_\rho K,
\end{cases}
\end{equation}
where
\[
N_\rho K = \set*{\nabla p}{
  p \in H^1(\Omega),\ p\geq0,\ p(1-\rho) = 0 }.
\]
This can be seen by observing that $T_\rho K$ contains all
  divergence-free vector fields and hence, by Helmoltz decomposition, any
  vector $w$ in the dual cone $N_\rho K$ is such that $\rho w$ is the gradient of a scalar function. Once we identify that this function should vanish on $\{\rho<1\}$, the constraints on the pressure $p$ are obtained by dualizing
  those defining $T_\rho K$.  More explicitly,
\begin{equation}\label{eq:cm-pde-press}
\begin{cases}
\partial_t \rho + \div(\rho v) = 0, & \rho\in K, \\
v = -\nabla V - \nabla p, \\
p\geq0,\ p(1-\rho) = 0,
\end{cases}
\end{equation}
where $p$ is to be thought of as a pressure field enforcing the
density constraint.  Finally, solutions to \eqref{eq:cm-pde-proj} also
formally coincide with the gradient flow
\begin{equation}\label{eq:cm-gf}
\begin{cases}
\partial_t \rho \in -\partial\E(\rho), \\
\rho(0) = \rho_0,
\end{cases}
\end{equation}
with respect to the Wasserstein distance of the energy
 $\E:\Prob(\Omega)\to\setR\cup\{\infty\}$ given by
$$ E(\rho) = \int V \dd \rho + F(\rho),  \hbox{ where } F(\rho) = i_K(\rho) = \begin{cases}
  0 & \hbox{ if } \rho \in K \\
  +\infty& \hbox{ if not.} \end{cases}
$$ The equation \eqref{eq:cm-gf} is a differential inclusion, meaning
that a priori the element of the subdifferential which is selected is
not known. This translates into the fact that the pressure only
belongs to a cone and is not an explicit function of the density, as
it happens in other equations, for instance of porous medium
type. However, in many differential inclusion problems there is indeed
uniqueness, corresponding to the idea that only choice of the pressure
will preserve the constraint: at a very formal level, we can say that
even if we only imposed $\nabla\cdot v\geq 0$ on the saturated region
--- and this was done in order not to increase the density in the
future after saturation --- we should also have $\nabla\cdot v= 0$, at
least for $t>0$, otherwise the density would violate the constraint in
the past. This means that $p$ should be uniquely determined as the
solution of $-\Delta p=\Delta V$ on the saturated region $\{\rho=1\}$,
with Dirichlet boundary condition on the part of the boundary of this
region contained in the interior of $\Omega$ and suitable
non-homogeneous Neumann condition $\nabla p\cdot n=-\nabla V\cdot n$
on the part contined in $\partial\Omega$.


\paragraph{Discretization} As explained in the introduction,
our strategy for the numerical solution of the crowd motion is to
employ a Lagrangian discretization  of \eqref{eq:cm-gf},
meaning that we consider the time evolution of a probability measure
which remains uniform over a set containing $N$ points:
\[
\Prob_N(\setR^d) = \set*{\frac1N\sum_{i=1}^N \delta_{x_i}}{x_i\in\setR^d}.
\]
Since the intersection between the constraint set $K=
\set{\rho\in\Probac(\Omega)}{\rho \leq 1}$ and $\Prob_N(\setR^d)$ is
empty, we are forced to replace the constraint $\rho\in K$ with a
penalization, therefore considering the regularized energy given by
\[
\E_\eps(\rho) = \int_\Omega V \d\rho
	+ \frac{1}{2\eps}\dist_K^2(\rho)
= \int_\Omega V \d\rho
+ \frac{1}{2\eps} \min_{\sigma\in K} W_2^2(\rho, \sigma).
\]
Note that $\frac{1}{2\eps} \dist_{K}^2$ is the Moreau-Yosida
regularization of the convex indicator function $F = i_K$ of the set $K$.

The evolution of the discrete measures is dealt with by keeping track
of the positions of the particles $X=(x_1,\dotsc,x_N)\in\setR^{Nd}$,
to which corresponds the associated measure $\mu_X = \frac1N
\sum_{i=1}^N \delta_{x_i}$.  Thanks to the correspondence between $X$
and $\mu_X$, we can think of $\E_\eps$ as an energy on the space of
particle positions too, given by
\[
\E_\eps(X) = \E_\eps(\mu_X)
= \frac1N \sum_{i=1}^N V(x_i) + \frac1{2\eps}\dist^2_K(\mu_X).
\]
Assume for a moment that the point set $X \in \Rsp^{Nd}$ does not belong to the set
   \begin{equation}\label{eq:DN}
    \mathcal{D}_N = \{ (x_1,\hdots,x_N)\in\Rsp^{Nd}\mid x_i=x_j \hbox{
      for some } i\neq j\}.
  \end{equation}
Then it is easy to see that for small perturbations $V = (v_1,\hdots,v_N)
\in \Rsp^{Nd}$, the optimal transport map
between $\mu_X$ and $\mu_{X+ V}$ simply maps $x_i$ to $x_i + v_i$,
thus showing that the map $X\mapsto \mu_{X}$ is locally isometric,
i.e. if $\nr{V}$ is small enough,
\[  \Wass_2^2(\mu_{X}, \mu_{X+V}) = \frac{1}{N} \nr{V}^2. \]
This suggests to replace the continuous Wasserstein gradient flow
  {\eqref{eq:cm-gf}} by the discrete gradient flow {\eqref{eq:cm-ode}}
  with respect to the Euclidean metric: 
\begin{equation}\label{eq:cm-ode}
\begin{cases}
\frac1N \dot X^N(t) = -\nabla\E^N\bigl(X^N(t)\bigr), \\
X^N(0) = X^N_0,
\end{cases}
\end{equation}
where $\E^N = \E_{\eps_N}$ for a suitable choice of $\eps_N\to0$. The
initial condition can be selected by optimal quantization as in \eqref{eq:initial}.

Given $\mu\in\Prob_N(\setR^d)$, let $\sigma\in K$ be the projection of $\mu$
onto $K$, i.e.\ a minimizer of
\[
\min_{\sigma\in K} W_2^2(\mu, \sigma).
\]
Its existence follows by compactness, while its uniqueness and
continuity with respect to $\mu$ are guaranteed by Proposition 5.2 in \cite{DePhilippis2016}.  Let also $T:\Omega\to\setR^d$ be the
(unique) optimal transport map from $\sigma$ to $\mu$. The cell
$L_i=T^{-1}(x_i)$ represents the part of the mass of $\sigma$ which is
attached to the particle $x_i$.
Denoting by $\beta_i(X) = \dashint_{L_i} x\d\sigma(x) = \int_{L_i} x\d\sigma(x)/\int_{L_i} \d \sigma(x)$ the barycenter of the cell $L_i$,
\autoref{prop:derivative} gives
\[
\frac{\partial\E^N}{\partial x_i}(X) = \frac1N \nabla V(x_i)
	+ \frac1{N\eps_N} \bigl(x_i-\beta_i(X)\bigr),
\]
and therefore \eqref{eq:cm-ode} becomes more explicitly
\begin{equation}\label{eq:cm-ode-expl}
\begin{cases}
\displaystyle
\dot x_i^N(t) = -\nabla V\bigl(x_i^N(t)\bigr)
	+ \frac1{\eps_N} \bigl[\beta_i\bigl(X^N(t)\bigr) - x_i^N(t)\bigr], \\
X^N(0) = X^N_0.
\end{cases}
\end{equation}

Moreover, \autoref{prop:derivative} shows that
$\frac1{2\eps}\dist^2_K$ is $\frac1{2N\eps}$-concave, which proves
that the vector field $-\nabla\E_\eps(X)$ is well-defined a.e. and
provides several useful properties of the flow of this vector
field. In particular, following \cite{Santambrogio2017} (slightly
adapting the proof of Proposition 2.3), one can prove the existence,
for every initial datum, of a solution of $$X'(t)\in
-\partial^+\E_\eps(X(t)),$$
where $\partial^+$ is the superdifferential
of semiconcave functions. Solutions of this ODE satisfy a reverse
Gronwall inequality which provides $$|X_1(t)-X_2(t)|\geq
e^{-Ct}|X_1(0)-X_2(0)|,$$ opposite to what happens for the gradient
flows of semiconvex functions: this is not useful for uniqueness
purposes, but states that solutions cannot concentrate too much, and
in particular we obtain that for a.e. initial datum the flow avoids
for a.e. time the non-differentiability set. In particular, we have
existence of solutions of $X'(t)=-\nabla\E_\eps(X(t)),$ in the
almost everywhere sense and for almost every initial datum. Therefore
it is always possible to find $\mu_N(0)$ and $X_N$ that satisfy the
hypothesis of the following theorem. However, we insist that the goal
of the present paper is to show approximation results, and the
existence proof for the ``discrete'' problem with a finite number of
particles is not the core of our analysis, which explains why we do
not provide more details about the existence for a.e. intial datum.

\begin{theorem}[Convergence of the discrete scheme]\label{thm:convergence-crowd-motion}
For every $N\in\setN$, let $\eps_N\in(0,\infty)$ and $\mu_N(0)\in\Prob_N(\setR^d)$
be such that
\[
\frac1{\eps_N} W_2^2(\rho_0, \mu_N(0)) \leq C, \qquad
\lim_{N\to\infty}\eps_N = 0.
\]
Let $X_N\in C^1([0,T], \setR^{Nd})$ be a solution of \eqref{eq:cm-ode} or,
equivalently, \eqref{eq:cm-ode-expl} and let $\mu_N:[0,T]\to\Prob_N(\setR^d)$
be the corresponding curve of measures.
Assume that
\begin{equation} \label{eq:cm-boundpres}
\frac1{\eps_N^2} \int_0^T W_2^2\oleft(\sigma_N,\frac{1}{N}\sum_{i=1}^N\delta_{\beta_i(X^N)}\right) \d t
\leq C,
\end{equation}
for some constant independent of $N$ and that $\rho_0\in K$.
Then, as $N\to\infty$, and up to subsequences,  $\mu_N\to\rho$ in $C^0\bigl([0,T]; W_2(\setR^d)\bigr)$,
where $\rho$ is a weak solution to
\eqref{eq:cm-pde-press}.
\end{theorem}


\begin{proof}
For simplicity of notation, let us write $\beta_i^N(t)$ in place of
$\beta_i\bigl(X^N(t))\bigr)$.
Define the time dependent vector valued measures
\[
M_N(t) = \sum_{i=1}^N \dot x_i^N(t) \frac1N \delta_{x_i^N(t)} =
- \frac1N \sum_{i=1}^N \left[\nabla V\bigl(x_i^N(t)\bigr) + \frac1{\eps_N}
	\bigl(x_i^N(t) - \beta_i^N(t)\bigr)\right] \delta_{x_i^N(t)} .
\]
Define also the space-time measures $\mu_N\in\Meas_+([0,T]\times\Omega)$ and
$M_N\in\Meas([0,T]\times\Omega;\setR^d)$ given by
\[
\mu_N = \int_0^T \delta_t\otimes \mu_N(t) \d t , \qquad
M_N = \int_0^T \delta_t\otimes M_N(t) \d t.
\]
By construction they satisfy the equation 
\begin{equation}\label{eq:ce-N-t}
\partial_t \mu_N(t) + \div M_N(t) = 0,
\end{equation}
in a weak sense
because, for any $\zeta\in C^1_c(\setR^d)$, one has
\begin{equation*}
\begin{split}
\frac{\d}{\d t} \int_{\setR^d} \zeta \d\mu_N(t)
&= \frac{\d}{\d t} \left( \frac1N\sum_{i=1}^N \zeta\bigl(x_i^N(t)\bigr) \right)
= \frac1N\sum_{i=1}^N \nabla\zeta\bigl(x_i^N(t)\bigr) \cdot \dot x_i^N(t) \\
&= \int_{\setR^d} \nabla\zeta \cdot \d M_N(t) 
\end{split}
\end{equation*}
This also means that $(\mu_N,M_N)$ solve the continuity equation
\begin{equation}\label{eq:ce-N}
\partial_t \mu_N + \div M_N = 0
\end{equation}
in distributional sense.

The first step is showing that the measures $\mu_N$ admit a limit in some sense.
This is a consequence of the energy estimate
\begin{equation}\label{eq:M_N-energy-estimate}
\begin{split}
\int_0^T \int_\Omega \abs*{\frac{\d M_N(t)}{\d\mu_N(t)}}^2\d\mu_N(t) \dd t
&= \int_0^T \frac1N \sum_{i=1}^N \abs{\dot x_i^N(t)}^2 \d t \\
&= \int_0^T -\nabla\E^N\bigl(X^N(t)\bigr)\cdot \dot X^N(t)\d t \\
&= \int_0^T -\frac{\d}{\d t} \E^N\bigl(X^N(t)\bigr) \d t \\
&= \E^N\bigl(X^N(0)\bigr) - \E^N\bigl(X^N(T)\bigr)
\leq \E^N\bigl(X^N(0)\bigr) \leq C ,
\end{split}
\end{equation}
because then the Benamou-Brenier formula for the $W_2$ distance
\begin{equation}\label{eq:H1}
\begin{split}
W_2^2\bigl(\mu_N(t_0),\mu_N(t_1)\bigr) &\leq
\dashint_{t_0}^{t_1} \int_\Omega \abs*{(t_1-t_0)\frac{\d M_N(t)}{\d \mu_N(t)}}^2
	\d\mu_N(t) \d t \\
&\leq \left( \int_0^T \int_\Omega \abs*{\frac{\d M_N(t)}{\d\mu_N(t)}}^2
	\d\mu_N(t) \d t \right) \abs{t_1-t_0}
\end{split}
\end{equation}
shows that the functions $[0,T]\to\bigl(\Prob(\Omega),W_2\bigr):t\mapsto\mu_N(t)$
are equi-continuous, since they are all $1/2$-H\"older with the same constant.
Ascoli-Arzelà then ensures that $\mu_N\to\rho$ in $C([0,T],W_2(\Omega))$,
up to a subsequence.
In particular, $\mu_N\weakto\rho$ in $\Meas_+([0,T]\times\Omega)$.

Next, we show that also the family of measures $M_N$ admits a limit. Indeed,
\begin{equation*}
\begin{split}
\normtv{M_N} &= \int_0^T \normtv{M_N(t)}\d t
= \int_0^T \frac1N \sum_{i=1}^N \abs{\dot x_i^N(t)} \d t
= \frac1N \sum_{i=1}^N \int_0^T \abs{\dot x_i^N(t)} \d t \\
&\leq \frac1N \sum_{i=1}^N \sqrt T
	\left(\int_0^T \abs{\dot x_i^N(t)}^2 \d t \right)^{1/2} \\
&\leq \sqrt T \left(\frac1N \sum_{i=1}^N
	\int_0^T \abs{\dot x_i^N(t)}^2 \d t \right)^{1/2} \\
&\leq \sqrt{T \E^N\bigl(X^N(0)\bigr)} \leq \sqrt{TC} ,
\end{split}
\end{equation*}
and by compactness in the space of measures, they admit a weak limit
$M_N\weakto M$ in $\Meas([0,T]\times\Omega;\setR^d)$.

In particular, the weak convergence of $\mu_N$ and $M_N$ is sufficient
to pass to the limit \eqref{eq:ce-N} and infer that
\[
\partial_t\rho + \div M = 0.
\]
To show that $M\ll \rho$, we will use a few
properties of the Benamou-Brenier functional
$\mathcal{B}_2: \mathcal{M}([0,T]\times \Omega) \times
\mathcal{M}([0,T]\times \Omega,\Rsp^d) \to\Rsp\cup\{+\infty\}$. These
properties are found in Proposition~5.18 of \cite{otam}: (i) If
$M\ll \mu$, $\mathcal{B}_2(\mu,M) = \int_{[0,T]\times \Omega}
\abs*{\frac{\d M}{\d\mu}}^2\d\mu$.  (ii) The functional
$\mathcal{B}_2$ is lower semi-continuous wrt narrow convergence.
(iii) $\mathcal{B}_2(\mu,M) < +\infty$ only if $\mu\geq 0$ and $M\ll
\mu$.
In our case, using the fact that $M_N \ll \mu_N$ and previous computations, we get that
\[
\mathcal{B}_2(\mu_N,M_N) = \int_0^T \int_\Omega \abs*{\frac{\d M_N(t)}{\d\mu_N(t)}}^2\d\mu_N(t)\dd t
\]
is uniformly bounded. Then, by lower semi-continuity,
$\mathcal{B}_2(\rho,M)$ is finite. This implies, by the third property
of $\mathcal{B}_2$ that $M \ll \rho$.



Let now $\sigma_N$ be the projection of $\mu_N$ on $K$ and let $T_N:\Omega\to\setR^d$
be the optimal transport map $(T_N)_\#\sigma_N=\mu_N$.
Notice that
\begin{equation*}
\begin{split}
\frac1{2\eps_N} W_2^2\bigl(\mu_N(t),\sigma_N(t)\bigr)
&= \frac{1}{2\eps_N} \min_{\sigma\in K} W_2^2\bigl(\mu_N(t), \sigma\bigr)
\leq \E_{\eps_N}\bigl(\mu_N(t)\bigr) \leq \E_{\eps_N}\bigl(\mu_N(0)\bigr) \\
&\leq \int_\Omega V\d\mu_N(0) + \frac1{2\eps_N} W_2^2(\rho_0,\mu_N(0)) \\
&\leq \int_\Omega V\d\rho_0 + \Lip(V)W_2(\rho_0,\mu_N(0))
	+ \frac1{2\eps_N} W_2^2(\rho_0,\mu_N(0))
\leq C,
\end{split}
\end{equation*}
therefore $\sigma_N(t)$ converges to the same limit $\rho(t)$ as
$\mu_N(t)$. We used that $\rho_0\in K$ so that $\dd_K(\rho_0) = 0$.
In particular this means that $\rho(t)\in K$ for all $t$.

For $\xi\in C^1(\setR^d;\setR^d)$ and omitting time dependence for brevity, we have
\begin{equation}\label{eq:MNprop}
\int_{\setR^d} \xi\cdot\d M_N =
	- \int_\Omega \left( \nabla V(T_N) + \frac{T_N-\Id}{\eps_N} \right)
	\cdot \xi(T_N) \d\sigma_N .
\end{equation}

By Brenier's Theorem and the particular structure of the optimal partial transport
problem, we have that
\[
T_N = \Id - \nabla\phi_N
\]
where $\phi_N:\Omega\to\setR$ is a semi-concave function satisfying
$\phi_N\leq0$ and $(1-\sigma_N)\phi_N=0$; see \cite[Lemma 3.1]{gf-cm}
or Proposition~\ref{prop:derivative-crowd} of this paper. If we introduce the
pressure field
\[
p_N = -\frac{\phi_N}{\eps_N} \geq 0
\]
we have that
\[
\frac{T_N-\Id}{\eps_N} = \nabla p_N, \qquad (1-\sigma_N)p_N=0.
\]

We must show that $p_N\weakto p$ in $L^2\bigl([0,T]; H^1(\Omega)\bigr)$ to some admissible pressure field. This follows from the equi-boundedness
\[
\int_0^T\int_\Omega \abs{\nabla p_N}^2 \d x \d t \leq C < \infty,
\]
together with a Poincar\'e inequality based on the fact that $|\{p_N=0\}|\geq |\Omega|-1$. This allows to transform $L^2$ bounds on the gradients into full $H^1$ bounds. We now prove the aforementioned equi-boundedness. For every
$t\in[0,T]$, which we omit for brevity of notation, we have that
\begin{equation}\label{eq:w2decomp}
\begin{split}
W_2^2\oleft(\mu_N,\sigma_N\right)
&= \sum_{i=1}^N \int_{L_i} \abs{y-x_i}^2 \d\sigma_N(y)
= \sum_{i=1}^N \int_{L_i} \abs{y-\beta_i+\beta_i-x_i}^2 \d\sigma_N(y) \\
&= \sum_{i=1}^N \int_{L_i} \abs{y-\beta_i}^2 \d\sigma_N(y)
	+ \sum_{i=1}^N \frac1N \abs{\beta_i-x_i}^2
	+ 2\sum_{i=1}^N \int_{L_i} \scal{y-\beta_i}{\beta_i-x_i} \d\sigma_N(y) \\
&\leq W_2^2\oleft(\sigma_N,\frac1N\sum_{i=1}^N\delta_{\beta_i}\right)
	+ \frac1N \sum_{i=1}^N \abs*{x_i-\beta_i}^2
\end{split}
\end{equation}
where in the last step we use that $\int_{L_i} (y-\beta_i) \d\sigma_N(y) = 0$.

The first term can be treated with the bound given by Assumption \eqref{eq:cm-boundpres}. 
For the second term, notice that
\begin{equation}\label{estimate from x'}
\begin{split}
\frac1{\eps_N^2} \int_0^T \frac1N \sum_{i=1}^N \abs*{x_i-\beta_i}^2 \dd t 
&= \int_0^T \frac1N \sum_{i=1}^N \abs*{\frac{x_i-\beta_i}{\eps_N}}^2 \d t \\
&= \int_0^T \frac1N \sum_{i=1}^N \abs{\dot x_i^N+\nabla V(X_i^N)}^2 \d t \\
&\leq 2\int_0^T \frac1N \sum_{i=1}^N
	\bigl(\abs{\dot x_i^N}^2 + \abs{\nabla V(X_i^N)}^2\bigr) \d t \\
&\leq 2 C + 2 \Lip(V)^2T
\end{split}
\end{equation}
by \eqref{eq:M_N-energy-estimate}.
In conclusion,
\begin{equation*}
\begin{split}
\int_0^T\int_\Omega \abs{\nabla p_N}^2 \d x \d t
&= \frac1{\eps_N^2} \int_0^T\int_\Omega \abs{\nabla \phi_N}^2 \d x \d t \\
&= \frac1{\eps_N^2} \int_0^T W_2^2\oleft(\mu_N,\sigma_N\right) \d t
\leq CT + 2C + 2\Lip(V)^2T.
\end{split}
\end{equation*}

The next step is to show that $p(1-\rho)=0$.
The difficulty is that both $\sigma_N$ and $p_N$ are converging weakly,
which is not sufficient in order to pass to the limit the nonlinear relation
$p_N(1-\sigma_N)=0$.

For $0\leq t_0<t_1\leq T$, let us introduce the average pressure
\[
p_N^{t_0,t_1}(x) = \dashint_{t_0}^{t_1} p_N(t,x) \d t
= \frac1{t_1-t_0} \int_{t_0}^{t_1} p_N(t,x) \d t .
\]
Define also the measures $\lambda_N\in\Meas_+([0,T])$ given by
\[
\lambda_N = \norm{\nabla p_N(t)}_{L^2(\Omega)} \cdot \leb^1
= \left(\int_\Omega \abs{\nabla p_N(t)}^2 \d x \right) \cdot \leb^1.
\]
Their total masses
\[
\abs{\lambda_N}([0,T]) = \int_0^T \int_\Omega \abs{\nabla p_N}^2 \d x \d t = \int_0^T\int_\Omega \abs{\nabla p_N}^2 \d\sigma_N \d t
\]
are uniformly bounded, as previously shown; therefore, up to a
subsequence, they converge weakly to a measure
$\lambda\in\Meas_+([0,T])$. Note that the second equality in the
previous formula uses that $p_N(1-\sigma_N)=0$.

For every $N$, we split in two pieces the following identity:
\begin{equation*}
\begin{split}
0 &= \dashint_{t_0}^{t_1} \int_\Omega p_N(t) \d(1-\sigma_N(t)) \d t \\
&= \dashint_{t_0}^{t_1} \int_\Omega p_N(t) \d(1-\sigma_N(t_0)) \d t
	+ \dashint_{t_0}^{t_1} \int_\Omega p_N(t) \d(\sigma_N(t_0)-\sigma_N(t)) \d t \\
&= \int_\Omega p_N^{t_0,t_1} \d(1-\sigma_N(t_0))
	+ \dashint_{t_0}^{t_1} \int_\Omega p_N(t) \d(\sigma_N(t_0)-\sigma_N(t)) \d t .
\end{split}
\end{equation*}
In order to deal with the first integral, we observe that we have strong convergence $p_N^{t_0,t_1}\xrightarrow{L^2(\Omega)}p^{t_0,t_1}$ since $p_N^{t_0,t_1}$ is bounded in $H^1$ (as it is obtained as an average, for fixed $t_0,t_1$ of a time-dependent function on which we have $L^2_t H^1_x$ bounds), and that we have
$\sigma_N(t_0)\weakto\rho(t_0)$ as $N\to\infty$ (this convergence is a weak convergence of measures, but it is also weak in $L^2$ because of the $L^\infty$ bounds on the densities). Hence, the first integral converges to
\[
\lim_{N\to\infty} \int_\Omega p_N^{t_0,t_1} \d(1-\sigma_N(t_0)) =
\int_\Omega p^{t_0,t_1} \d(1-\rho(t_0)) .
\]
At any Lebesgue point $t_0$ of the map $[0,T]\to L^2(\Omega):t\mapsto p(t)$ we have
$p^{t_0,t_1}\xrightarrow[t_1\to t_0]{} p(t_0)$, hence
\[
\lim_{t_1\to t_0} \int_\Omega p^{t_0,t_1} \d(1-\rho(t_0))
= \int_\Omega p(t_0) \d(1-\rho(t_0)) .
\]


Employing \autoref{lem:H1-W2-inequality}, the second integral can be estimated as
\begin{equation*}
\begin{split}
\hspace{3cm}&\hspace{-3cm}
\abs*{\dashint_{t_0}^{t_1} \int_\Omega p_N(t) \d(\sigma_N(t_0)-\sigma_N(t)) \d t}
\leq \dashint_{t_0}^{t_1} \norm{\nabla p_N(t)}_{L^2(\Omega)}
	W_2\bigl(\sigma_N(t_0), \sigma_N(t)\bigr) \d t \\
&\leq \omega(t_1-t_0) \dashint_{t_0}^{t_1}
	\norm{\nabla p_N(t)}_{L^2(\Omega)} \d t \\
&\leq \omega(t_1-t_0)  \left( \dashint_{t_0}^{t_1}
	\norm{\nabla p_N(t)}_{L^2(\Omega)}^2 \d t \right)^{1/2} \\
&= \omega(t_1-t_0) \sqrt{\frac{\lambda_N([t_0,t_1])}{t_1-t_0}}, 
\end{split}
\end{equation*}
where $\omega$ is a continuity modulus for the curve $t\mapsto \sigma_N(t)$ in the Wasserstein space $W_2$. Note that the continuity of the curve $t\mapsto \sigma_N(t)$ comes from the continuity of $t\mapsto \mu_N(t)$ (a consequence of  \eqref{eq:H1}) and the continuity of the projection operator $\mu \mapsto \arg\min_\sigma \Wass_2^2(\mu,\sigma)$; see, for instance, \cite{DePhilippis2016}).

When $N\to\infty$, for almost every $t_0$ and $t_1$ we have
\[
\lim_{N\to\infty} \sqrt{\frac{\lambda_N([t_0,t_1])}{t_1-t_0}} \leq \sqrt{\frac{\lambda([t_0,t_1])}{t_1-t_0}},
\]
which tends to a finite constant when $t_1\to t_0$ for a.e. $t_0$. This latter fact  comes from differentiation of measures: this quantity represents the density of the measure $\lambda$ w.r.t. the Lebesgue measure on the real line, and the limit exists and is equal to the density of the absolutely continuous part of $\lambda$ a.e. If we also consider the factor $\omega(t_1-t_0) $, we see that the second integral goes to $0$ for almost every $t_0$ when taking the limits $N\to\infty$ and $t_1\to t_0$, in this order.

Summing up, we have shown that
\[
\int_\Omega p(t_0) \d(1-\rho(t_0)) =
\lim_{t_1\to t_0} \lim_{N\to\infty} \dashint_{t_0}^{t_1}\int_\Omega p_N(t_0) \d(1-\sigma_N(t_0))\dd t = 0
\]
for almost every $t_0\in[0,T]$, which proves $p(1-\rho)=0$, by the positivity of $p$.

We can finally show that $M=(-\nabla V - \nabla p)\rho$. Fix $\xi\in C^1([0,T]\times\Omega;\setR^d)$. By \eqref{eq:MNprop}, we know that
\[
\int_0^T\int_{\Omega} \xi\cdot\d M_N =
	- \int_0^T\int_\Omega \left( \nabla V(T_N) + \nabla p_N \right)
	\cdot \xi(T_N) \d\sigma_N \d t.
\]
The first term passes to the limit because
\[
\int_0^T \int_\Omega \nabla V(T_N) \cdot \xi(T_N) \d\sigma_N  = \int_0^T \int_\Omega \nabla V \cdot \xi \d\mu_N \to
\int_0^T \int_\Omega \nabla V \cdot \xi \d\rho.
\]
For the second term,
\begin{equation*}
\begin{split}
\abs*{\int_0^T\int_\Omega \nabla p_N \cdot [\xi(T_N)-\xi] \d\sigma_N \d t}
&\leq \left(\int_0^T\int_\Omega \abs{\nabla p_N}^2 \d\sigma_N \d t\right)^{1/2} \\
&\hspace{2cm} \cdot\left(\int_0^T\int_\Omega \abs{\xi(T_N)-\xi}^2 \d\sigma_N \d t\right)^{1/2} \\
&\leq \left(\int_0^T\int_\Omega \abs{\nabla p_N}^2 \d\sigma_N \d t\right)^{1/2}
	\Lip(\xi) W_2(\mu_N,\sigma_N) \to 0,
\end{split}
\end{equation*}
therefore
\begin{equation*}
\begin{split}
\lim_{N\to\infty} \int_0^T\int_\Omega \nabla p_N \cdot \xi(T_N) \d\sigma_N \d t
&= \lim_{N\to\infty} \int_0^T\int_\Omega \nabla p_N \cdot \xi \d\sigma_N \d t \\
&= \lim_{N\to\infty} \int_0^T\int_\Omega \nabla p_N \cdot \xi \d x \d t \\
&= \int_0^T\int_\Omega \nabla p \cdot \xi \d x \d t \\
&= \int_0^T\int_\Omega \nabla p \cdot \xi \d \rho \d t . \qedhere
\end{split}
\end{equation*}
\end{proof}

The following lemma is borrowed from \cite[Lemma 3.5]{gf-cm} (but was first presented in other papers, such as \cite{Loeper}).

\begin{lemma}\label{lem:H1-W2-inequality}
Let $\mu_0,\mu_1\in\Prob(\Omega)$ be probability measures with densities
bounded by $1$. Then for all $f\in H^1(\Omega)$ we have that
\[
\abs*{\int_\Omega f \d(\mu_1-\mu_0)} \leq
\norm{\nabla f}_{L^2(\Omega)} W_2(\mu_0,\mu_1) .
\]
\end{lemma}


\begin{remark}
The convergence result can be generalized to handle PDEs involving
other terms such as self-interaction involving a $\mathcal{C}^{1,1}$ kernel
$W$:
\[
\begin{cases}
\partial_t \rho + \div(\rho v) = 0, & \rho\in K, \\
v = -\nabla V - \nabla W*\rho - w, & w\in N_\rho K,
\end{cases}
\]
which can be regarded as the gradient flow, in $(\Prob(\Omega), \Wass_2)$ of the energy
\[
\E(\rho) = \begin{cases}
\int_\Omega V(x) \d\rho(x) + \frac 12\int_\Omega\int_\Omega W(x-y) \d\rho(x)\d\rho(y)
&\hbox{if } \rho \in K, \\
\infty &\text{otherwise},
\end{cases}
\]
provided the Kernel $W$ is even. 
In this case, the discretized system becomes
\[
\begin{cases}
\displaystyle
\dot x_i^N(t) = -\nabla V\bigl(x_i^N(t)\bigr)
	- \frac1N \sum_{j} \nabla W(x_i-x_j)
	+ \frac1{\eps_N} \bigl[\beta_i\bigl(X^N(t)\bigr) - x_i^N(t)\bigr], \\
X^N(0) = X^N_0.
\end{cases}
\]


The only difference with respect to Theorem \ref{thm:convergence-crowd-motion} will the presence of a modified velocity field $\nabla V_N$ instead of $\nabla V$, with $V_N=V+W*\mu_N$. In the proof of Theorem \ref{thm:convergence-crowd-motion} we used the weak convergence of $\mu_N$ to handle the term $\int \nabla V\cdot \xi d\mu_N$; we will now also need the uniform convergence $\nabla V_N\to \nabla (V+W*\rho)$ (which is a consequence of the regularity assumption on $W$) to handle the same term. Also note that in the definition of the flow one can omit the term $\nabla W (x_i-x_j)$ for $i=j$, as it is usually done, since anyway $\nabla W(0)=0$.
\end{remark}

\begin{theorem}[Convergence of the discrete scheme in 1D]\label{thm:convergence-crowd-motion-1d}
Let $\Omega\subset\setR$ be an interval. For every $N\in\setN$, let $\eps_N=1/N$ and $\mu_N(0)\in\Prob_N(\setR)$
be such that
\[
\frac1{\eps_N} W_2^2(\rho_0, \mu_N(0)) \leq C, 
\]
Let $X_N\in C^1([0,T], \setR^N)$ be a solution of \eqref{eq:cm-ode} or,
equivalently, \eqref{eq:cm-ode-expl} and let $\mu_N:[0,T]\to\Prob_N(\setR)$
be the corresponding curve of measures.
Then, as $N\to\infty$, and up to subsequences,  $\mu_N\to\rho$ in $C^0\bigl([0,T]; W_2(\setR)\bigr)$,
where $\rho$ is a weak solution to
\eqref{eq:cm-pde-press}.
\end{theorem}

\begin{proof}
This is an immediate consequence of \autoref{thm:convergence-crowd-motion} and \autoref{prop:pressure-bound-crowd-motion}, which allows to verify the assumption
\[
\frac1{\eps_N^2} \int_0^T W_2^2\oleft(\sigma_N,\frac1N\sum_{i=1}^N\delta_{\beta_i}\right) \d t
\leq C. \qedhere
\]
\end{proof}

%% file: sec3-diffusion.tex
\section{Lagrangian discretization of linear diffusion}

The previously presented Lagrangian scheme can be adapted to solve also the
advection-diffusion equation
\begin{equation}\label{eq:adv-diff-pde}
\begin{cases}
\partial_t \rho + \div(\rho v) = 0, & \\
v = -\nabla V - \nabla \log\rho, & 
\end{cases}
\end{equation}
on a bounded domain $\Omega$, with no-flux boundary conditions.
This equation arises as the gradient flow with respect to $W_2$ of the energy
\[
\E(\rho) = \int V\dd \rho + H(\rho) \hbox{ where } H(\rho) = \begin{cases}
\int_\Omega \log\rho\d\rho
	& \rho \ll \leb^d\res\Omega, \\
\infty & \text{otherwise},
\end{cases}
\]
We adopt the same Lagrangian discretization as before.  For the atomic
measures $\mu_N$, the entropy (the second term in the energy) is
identically $+\infty$, therefore we need to  substitute it with a
similar functional, in the same manner that we replaced the hard
constraint $\rho\in K$ with a penalization. To this end, we consider
its Moreau-Yosida regularization
\begin{equation}\label{eq:Heps}
H_\eps(\rho) = \min_{\sigma\in\Prob(\Omega)} \frac1{2\eps} W_2^2(\rho,\sigma) + H(\sigma).
\end{equation}
and the new energy becomes $E_\eps(\rho) = \int_\Omega V\d\rho +
H_\eps(\rho).$ Letting $F_\eps(x_1,\hdots,x_N) = H_\eps(\mu_X)$, the
discrete measure $\mu_N(t)$ represented by the particles $X_N(t)$ can
then evolve according to the system of ODE as before, namely
\begin{equation}\label{eq:adv-diff-ode}
\begin{cases}
\frac{1}{N} \dot{x}_i(t) = - \frac{1}{N} \nabla V(x_i) - \nabla_{x_i} F_\eps(X), \\
X^N(0) = X^N_0.
\end{cases}
\end{equation}

\begin{theorem}[Convergence of the discrete scheme]\label{thm:conv-diffusion}
For every $N\in\setN$, let $\eps_N\in(0,\infty)$ and $\mu_N(0)\in\Prob_N(\Rsp^d)$ such that 
\[
\frac1{\eps_N} W_2^2(\rho_0, \mu_N(0)) \leq C, \qquad
\lim_{N\to\infty}\eps_N = 0.
\]
Let $X_N\in C^1([0,T], \setR^{Nd})$ be a solution of \eqref{eq:adv-diff-ode}
and let $\mu_N:[0,T]\to\Prob_N(\setR^d)$ be the corresponding curve of measures.
Assume that
\begin{equation}\label{eq:assump-diff}
\frac1{\eps_N} \int_0^T W_2^2\oleft(\sigma_N,\frac1N\sum_{i=1}^N\delta_{\beta_i}\right) \d t \to 0,
\end{equation}
where $\sigma_N$ is the minimizer in the definition of $H_\eps(\mu_N)$ (see \eqref{eq:Heps}).
Then, as $N\to\infty$, and up to subsequences,  $\mu_N\to\rho$ in $C^0\bigl([0,T]; W_2(\setR^d)\bigr)$,
where $\rho$ is a weak solution to
\eqref{eq:adv-diff-pde}.
\end{theorem}

\begin{proof}
Define as before the vector measures
\[
M_N = \sum_{i=0}^N \dot x_i^N(t) \frac1N \delta_{x_i^N(t)}.
\]
Together with $\mu_N$, they solve the continuity equation
\[
\partial_t \mu_N(t) + \div M_N(t) = 0.
\]
Moreover,
\begin{equation*}
\begin{split}
\int_0^T \abs*{\frac{\d M_N(t)}{\d\mu_N(t)}}^2\d\mu_N(t)
&= \int_0^T \frac1N \sum_{i=1}^N \abs{\dot x_i^N(t)}^2 \d t \\
&= \int_0^T -\nabla\E^N\bigl(X^N(t)\bigr)\cdot \dot X^N(t)\d t \\
&= \int_0^T -\frac{\d}{\d t} \E^N\bigl(X^N(t)\bigr) \d t \\
&= \E^N\bigl(X^N(0)\bigr) - \E^N\bigl(X^N(T)\bigr)
\leq C,
\end{split}
\end{equation*}
where we used that the domain is bounded to see that the entropy is
bounded from below by some constant, implying
$\E^N\bigl(X^N(0)\bigr)\geq C$. This shows that the functions
$[0,T]\to\bigl(\Prob(\Omega),W_2\bigr):t\mapsto\mu_N(t)$ are
equi-continuous, because they are $1/2$-H\"older with the same
constant.  Ascoli-Arzelà then ensures that $\mu_N\weakto\rho$, up to a
subsequence.

The rest of the proof is similar to the previous one with the following modifications.

The measure $\sigma_N$ minimizing
\[
\min_{\sigma\in\Prob(\Omega)} \int_\Omega \log\sigma\d\sigma
	+ \frac1{2\eps} W_2^2(\mu_N,\sigma)
\]
satisfies
\[
\sigma_N = c_N e^{-\phi_N/(2\eps_N)} \leb^d\res\Omega.
\]
where $\phi_N$ is the optimal potential from $\sigma_N$ to $\mu_N$ and $c_N$ is a normalization constant. This optimality condition can be recovered from the first variation of the objective functional. This implies that $\leb^d$-almost everywhere on $\Omega$ the density $\sigma_N$ is strictly positive and
\[
\frac{T_N-\Id}{\eps_N} = -\frac{\nabla \phi_N}{2\eps_N} = \nabla\log\sigma_N
= \frac{\nabla\sigma_N}{\sigma_N}.
\]

Passing to the limit $M_N$ in order to get $M=-\nabla V\rho-\nabla\rho$ is now easier because
\[
\int_0^T \int_\Omega \nabla V \cdot \xi \d\mu_N \to
\int_0^T \int_\Omega \nabla V \cdot \xi \d\rho
\]
as before. Setting $p_N=\log\sigma_N$ (which is the term which plays a similar role to that of the pressure in the previous section), we have for any $\xi \in \mathcal{C}^0_c(\Omega)$,
\begin{equation*}
  \begin{split}
\lim_{N\to\infty} \int_0^T \int_\Omega \nabla p_N \cdot \xi(T_N) \d\sigma_N \d t
&= \lim_{N\to\infty} \int_0^T \int_\Omega \nabla p_N \cdot \xi \d\sigma_N \d t \\
&= \lim_{N\to\infty} \int_0^T \int_\Omega \nabla \sigma_N \cdot \xi \d x \d t \\
&= -\lim_{N\to\infty} \int_0^T \int_\Omega \sigma_N \div \xi \d x \d t\\
&= -\int_0^T \int_\Omega \rho \div(\xi) \d x \d t.
\end{split}
\end{equation*}
The first step in the above equation is justified because
\begin{equation*}
\begin{split}
\abs*{\int_0^T \int_\Omega \nabla p_N \cdot \bigl(\xi(T_N)-\xi\bigr) \d\sigma_N \d t}
&\leq \Lip(\xi)\ \int_0^T \int_\Omega \abs{\nabla p_N} \cdot \abs{T_N-\Id} \d\sigma_N \d t \\
&= \Lip(\xi) \frac1{\eps_N} \int_0^T \int_\Omega \abs{T_N-\Id}^2 \d\sigma_N \d t \\
&= \Lip(\xi) \frac1{\eps_N} \int_0^T W_2^2\oleft(\sigma_N,\mu_N\right) \d t
\to 0.
\end{split}
\end{equation*}
The last term tends to $0$ by writing, as in \eqref{eq:w2decomp},
$$W_2^2\oleft(\sigma_N,\mu_N\right) \leq W_2^2\oleft(\sigma_N,\frac1N\sum_{i=1}^N\delta_{\beta_i}\right)
	+  \frac1N \sum_{i=1}^N \abs*{x_i-\beta_i}^2 .$$
	The first term tends to $0$ by assumption, and the second term is $O(\eps_N)$ because of \eqref{estimate from x'}. \end{proof}

\begin{theorem}[Convergence of the discrete scheme in 1D]\label{thm:conv-diffusion-1d}
Let $\Omega\subset\setR$ be a bounded interval. For every $N\in\setN$, take a number $\eps_N>0$ and $\mu_N(0)\in\Prob_N(\setR)$ such that
\[
\frac1{\eps_N} W_2^2(\rho_0, \mu_N(0)) \leq C, \qquad
\lim_{N\to\infty}\eps_N = 0, \quad \lim_{N\to\infty}N^2\eps_N = +\infty.
\]
Let $X_N\in C^1([0,T], \setR^N)$ be a solution of \eqref{eq:adv-diff-ode}
and let $\mu_N:[0,T]\to\Prob_N(\setR)$ be the corresponding curve of measures.
Then, as $N\to\infty$, and up to subsequences, $\mu_N\to\rho$ in $C^0\bigl([0,T]; W_2(\setR)\bigr)$,
where $\rho$ is a weak solution to
\eqref{eq:adv-diff-pde}.
\end{theorem}

\begin{proof}
This is an immediate consequence of \autoref{thm:conv-diffusion} and \autoref{prop:pressure-bound-diffusion}, which provide
\[
\frac1{\eps_N} \int_0^T W_2^2\oleft(\sigma_N,\frac1N\sum_{i=1}^N\delta_{\beta_i}\right) \d t\leq \frac{1}{N\sqrt{\eps_N}} \to 0. \qedhere
\]
\end{proof}

%% file: sec4-pressure.tex
\section{Bounds in 1D}
\label{sec:bounds}

In this section we prove that, for both the crowd motion and the linear diffusion discretizations, in one dimension there are bounds on the quantities which are relevant for the application of Theorems \ref{thm:convergence-crowd-motion} and \ref{thm:conv-diffusion}. The results come from a static analysis, in the sense that the evolution equations do not play any role in the estimates.

We begin with the easier case of the crowd motion.


\begin{proposition}\label{prop:pressure-bound-crowd-motion}
Let $\Omega\subset\setR$ be an interval. Let $x_1,\dots,x_N\in\setR$,
let $\mu_N$ be the corresponding atomic measure, $\sigma_N$ its
$W_2$-projection on $\set{\rho\in\Prob(\Omega)}{\rho\leq1}$, $T$ the
optimal transport map between $\sigma_N$ and $\mu_N$. Define the Laguerre cells and their barycenters as
$$ L_i = T^{-1}(\{x_i\}) $$
$$ \beta_i = N \int_{L_i} x \dd\sigma_N(x) $$
Then, choosing $\eps=1/N$,
\[
\frac1{\eps^2}  W_2^2\oleft(\sigma_N,\frac{1}{N}\sum_{i=1}^N\delta_{\beta_i}\right) \d t
\leq \frac{1}{12}.
\]
\end{proposition}

\begin{proof}
Each Laguerre cell $L_i$ is an interval of length $1/N$ and its barycenter is the midpoint. Moreover, $\sigma_N$ has constant density $1$ on every cell $L_i$, therefore
\begin{equation*}
\begin{split}
W_2^2\oleft(\sigma_N,\sum_{i=1}^N\delta_{\beta_i}\right)
&= \sum_{i=1}^N \int_{L_i} \abs{y-\beta_i}^2 \d y
= \sum_{i=1}^N \int_{\beta_i-1/(2N)}^{\beta_i+1/(2N)} \abs{y-\beta_i}^2 \d y \\
&= N \int_{-1/(2N)}^{1/(2N)} y^2\d y = \frac1{12N^2},
\end{split}
\end{equation*}
which gives the claim.
\end{proof}

We now pass to the case which is relevant for linear diffusion.
\begin{proposition}\label{prop:pressure-bound-diffusion}
Let $\Omega\subset\setR$ be a bounded interval. Let $x_1,\dots,x_N\in\setR$, let $\mu_N$ be the corresponding atomic measure, and define, for $\eps>0$:
\[
\sigma_N = \argmin_{\rho\in\Prob(\Omega)} \frac1{2\eps}W_2^2(\mu_N,\rho) + \Ent(\rho).
\]
Let $\beta_1,\dots,\beta_N$ be the barycenters of the Laguerre cells $L_1,\dots,L_N$ of $\sigma_N$. Then we have
\[
W_2^2\oleft(\sigma_N,\frac1N\sum_{i=1}^N\delta_{\beta_i}\right)\leq C(\Omega)\frac{\sqrt{\eps}}{N},
\]
where $C(\Omega)$ only depends on the length $|\Omega|$ of $\Omega$.
\end{proposition}

The proof of this proposition relies on the next lemma, which is a
particular case of the main theorem of \cite{kanter1977reduction}.
\begin{lemma}\label{lem:quantization-long-gaussian-cells}
  Let $\sigma = e^{- (y-x)^2/(2\eps)} \dd y$ be a Gaussian measure,
  let $L$ be an interval such that $\sigma(L) = \frac{1}{N}$.  Then,
  the variance of $N \sigma \res L$ is upper bounded by the
  variance of $\sigma$:
  $$ N \int_{L} (y-\beta)^2 \dd \sigma(y) \leq
  \eps, \hbox{ where } \beta = N \int_L
  y\dd\sigma(y). $$
\end{lemma}

\begin{proof}[Proof of \autoref{prop:pressure-bound-diffusion}]
Let $\ell_i=\abs{L_i}$ denote the length of the $i$-th Laguerre cell.
We fix a parameter $\bar\ell\in(0,1)$ to be specified later and divide the cells in two groups:
\begin{itemize}
\item short cells: $S=\set{i}{\ell_i<\bar\ell}$;
\item long cells: $L=\set{i}{\ell_i\geq\bar\ell}$.
\end{itemize}
Notice that $\abs{S}\leq N$ and $\abs{L}\leq
\abs{\Omega}/\bar\ell$. By \autoref{prop:my-boltzmann}, we know that
the restriction of $\sigma_N$ to the Laguerre cell $L_i$ is proportional
to a Gaussian of the form
$\exp(-\frac{1}{2\eps}(x-x_i)^2)$. Therefore, using
\autoref{lem:quantization-long-gaussian-cells} to estimate the
contribution of the long cells, we get
\begin{equation*}
\begin{split}
W_2^2\oleft(\sigma_N,\frac1N\sum_{i=1}^N\delta_{\beta_i}\right)
&= \sum_{i\in S} \int_{L_i} (y-\beta_i)^2 \dd \sigma_N(y)
	+ \sum_{i\in L} \int_{L_i} (y-\beta_i)^2 \dd \sigma_N(y) \\
&\leq \sum_{i\in S} \frac1N \ell_i^2 + \sum_{i\in L} \frac{\eps}{2N}
\leq \frac{\bar\ell}{N} \sum_{i\in S}\ell_i + \frac{\eps\abs{\Omega}}{2\bar\ell N} \\
&\leq \frac{\bar\ell |\Omega|}{N} + \frac{\eps\abs{\Omega}}{2\bar\ell N}
= \frac{3\sqrt{\eps} \abs{\Omega}}{2N},
\end{split}
\end{equation*}
where in the last step we chose $\bar\ell=\sqrt{\eps}$.
\end{proof}


%% file: sec5-numerics.tex
\section{Numerical scheme}

\subsection{Computation of the Moreau-Yosida regularization}
\label{sec:comput-my}
Let $F: \Probac(\Omega)\to\setR \cup\{+\infty\}$, which we assume to
be lower-semicontinuous with respect to the Wasserstein metric
$\Wass_2$. We consider the Moreau-Yosida regularization
\begin{equation} \label{eq:my}
  F_\eps:X \in \setR^{Nd} \mapsto \inf_{\sigma \in \Probac(\Omega)}
  \frac{1}{2\eps}\Wass_2^2(\sigma,\mu_X) + F(\sigma),
\end{equation}
and we assume that for every $X\in \setR^{Nd}$, the minimization
problem defining $F_\eps(X)$ admits a unique solution. This  assumption  is satisfied in the two relevant cases for this paper, since the projection onto measures with bounded densities is always unique, see \cite{DePhilippis2016}, and the minimizer in the entropy case is unique because of strict convexity.  We let
$$\mathcal{D}_N = \{ (x_1,\hdots,x_N)\in\setR^{Nd} \mid x_i = x_j
\hbox{ for some } i\neq j\}.$$ Our first proposition gives an explicit
formulation for the gradient of $F_\eps$ given a solution $\sigma$
of the minimization problem defininig $F_\eps$. We recall that a
function $F$ on $\Rsp^k$ is \emph{$\lambda$-semi-concave} if and only
if $F - \lambda\nr{\cdot}^2$ is concave.

\begin{proposition} \label{prop:derivative}
   $F_\eps$ is $\frac{1}{2N\eps}$-semi-concave on $\setR^{Nd}$ and
  continuously differentiable on $\setR^{Nd} \setminus \mathcal{D}_N$. Given
  $X=(x_1,\hdots,x_N)\in\setR^{Nd} \setminus \mathcal{D}_N$, we let $\sigma$ the
   unique minimizer in \eqref{eq:my}, $T$ the unique optimal transport
   map between $\sigma$ and $\mu_X$, and $L_i = T^{-1}(x_i)$. Then,
   $$\nabla_{x_i} F_\eps(X) = \frac{1}{N} \frac{x_i - \beta_i(X)}{\eps} \hbox{ where } \beta_i(X) := N \int_{L_i} x \d\sigma(x)$$
\end{proposition}

  \begin{proof} First, let us underline that we work under the assumption that the optimal $\sigma$ in the minimization
problem defining $F_\eps(X)$ is unique for every $X$. This uniqueness implies continuity of the map $X\mapsto \sigma$. Let $X \in \setR^{Nd}$, $\sigma$ the
    unique minimizer in \eqref{eq:my} and $T$ the unique optimal
    transport map between $\sigma$ and $\mu_X$.  Given $Y \in
    \setR^{Nd}$,
    $$ F_\eps(Y) \leq \frac{1}{2\eps}\Wass_2^2(\sigma,\mu_Y) +
    F(\sigma).$$ By construction, one can decompose $\sigma =
    \sum_{1\leq i \leq N} \sigma_i$ where $\sigma_i\geq 0$,
    $T(\sigma_i) = x_i$, and such that
    $\sigma_i(\Omega)=1/N$. Considering the transport which maps
    $\sigma_i$-a.e. point to $y_i$, one gets
    \[
    \begin{aligned}
      F_\eps(Y) &\leq \frac{1}{2\eps} \sum_i \int \abs{x - y_i}^2 \d \sigma_i + F(\sigma) \\
      &\leq \frac{1}{2\eps} \sum_i \int \abs{x - x_i + x_i - y_i}^2 \d \sigma_i + F(\sigma) \\
      &\leq F_\eps(X) + \sum_i \sca{\frac{1}{N\eps}(\beta_i(X) - x_i)}{x_i - y_i} + \frac{1}{2N\eps} \sum_i \abs{x_i - y_i}^2\\
      &\leq F_\eps(X) + \sca{\frac{1}{N\eps}(\beta(X) - X)}{X-Y} + \frac{1}{2N\eps} \abs{X-Y}^2
    \end{aligned}
    \]
    where we have set $\beta_i(X) = N \int x \d\sigma_i$ and $\beta(X) = (\beta_1(X),\hdots,\beta_N(X))$. This inequality can be rewritten as
    $$ G_\eps(Y) \leq G_\eps(X) +
    \sca{-\frac{1}{N\eps}\beta(X)}{Y-X} $$ where $G_\eps(X) = F_\eps(X)
    - \frac{1}{2N\eps}\nr{X}^2$.  This shows that
    $-\frac{1}{N\eps}\beta(X)$ belongs to the superdifferential of
    $G_\eps$ at $X$, so that the superdifferential of $G_\eps$ is
    never empty, also showing that $G_\eps$ is concave.  If $X \not\in
    \mathcal{D}_N$, then $\sigma_i = \sigma \res T^{-1}(x_i)$ and the
    point $\beta_i(X)$ is uniquely defined (we use here the hypothesis
    on the uniqueness of the minimal $\sigma$ in \eqref{eq:my}). Using
    the stability of optimal transport maps, we get that $X \in
    \setR^{Nd}\setminus D \mapsto \beta(X)$ is continuous on
    $\setR^{Nd}\setminus D_N$, which shows that $G_\eps \in
    \mathcal{C}^{1}(\setR^{Nd}\setminus D_N)$ and that
    $$\nabla_{x_i}
    G_\eps(X) = -\frac{1}{N\eps} \beta_i(X).$$
The properties of $F_\eps$ can be deduced from those of $G_\eps$.
  \end{proof}

  The next two propositions explain how to compute the optimal
  $\sigma$ in the definition of the Moreau-Yosida regularization in
  the crowd motion and linear diffusion. Using Kantorovich duality,
  this problem can be reformulated as the computation of a Kantorovich
  potential satisfying a finite-dimensional non-linear system,
  \eqref{eq:crowd-my-dual} or \eqref{eq:ent-my-dual}.

  Given $x_1,\hdots,x_N \in \setR^d$ and $\psi\in \setR^N$ we define the
  Laguerre cell of the point $x_i$ with respect to $\Omega$ as
  $$ L_i(\psi) = \{ x \in \Omega \mid \forall j, \nr{x - x_i}^2 -
  \psi_i \leq \nr{x - x_j}^2 - \psi_j \}. $$ Note that the next
  proposition is a special case of the characterization of Wasserstein
  projections on $K$, proven in \cite[Lemma 3.1]{gf-cm}. It is
  illustrated in Figure~\ref{fig:crowd-gradient}.

        \begin{figure} \label{fig:crowd-gradient}
      \begin{centering}
      \includegraphics[width=.24\textwidth]{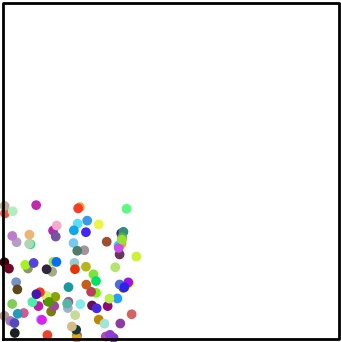}\hfill 
      \includegraphics[width=.24\textwidth]{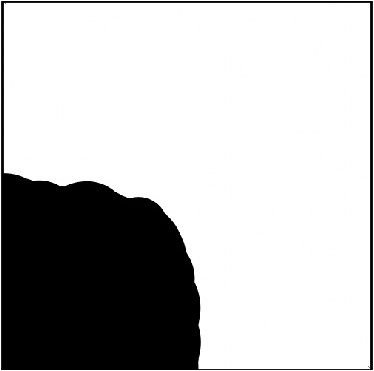}\hfill 
      \includegraphics[width=.24\textwidth]{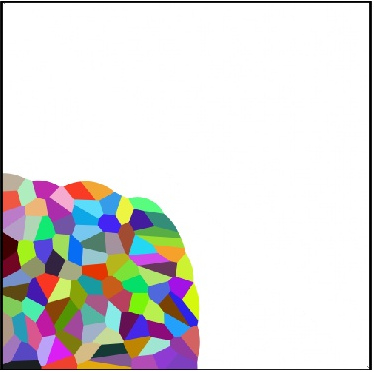}\hfill 
      \includegraphics[width=.24\textwidth]{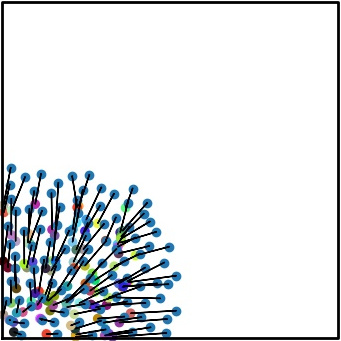}
      \end{centering}
      \caption{From left to right: a) a point cloud $x_1,\hdots,x_N$
        drawn uniformly in $[0,\frac{4}{5}]^2$ with $N=100$ points.
        b) the support of Wasserstein projection of $\mu =
        \frac{1}{N}\sum_i\delta_{x_i}$ on the set of probability
        densities bounded by $1$, c) the Laguerre cells d) the arrows
        joining colored points to blue points are proportional to
        $-\nabla_{x_i} \Ent_\eps(x_1,\hdots,x_N)$.}
    \end{figure}

  \begin{proposition}\label{prop:derivative-crowd}
    Consider $F:\Prob(\Omega)\to\setR$ defined by $F(\mu) = 0$ if
    $\mu\in K$ and $F(\mu) = +\infty$ otherwise, where $K$ is defined
    in \eqref{eq:K}. Then for all $X\in\setR^{Nd}\setminus
    \mathcal{D}_N$, there exists $\psi \in \Rsp^N_-$ such that
      \begin{equation} \label{eq:crowd-my-dual}
        \forall i,~ \abs{L_i(\psi)\cap \Omega\cap \Ball(x_i, \sqrt{\psi_i})} = \frac{1}{N}
      \end{equation}
      Given such a $\psi$, define $\phi = \min(\min_{i} \nr{\cdot - x_i}^2 - \psi_i,0)$ and $\sigma = \mathbf{1}_{\{\phi<0\} \cap \Omega}$.
      \begin{itemize}
      \item[(a)] $\sigma \in \Probac(\Omega)$ is the Wasserstein projection of $\mu_X$ on
        $K$,
      \item[(b)] $\phi\leq 0, $ $\phi(1-\sigma) = 0$
      \item[(c)]  $(\phi,\psi)$ is an admissible pair of Kantorovich potential in the transport from $\sigma$ to $\mu_X$
      \end{itemize}
  \end{proposition}
  
  \begin{proof}
    By Kantorovich duality, one can write for any $\sigma \in
    \Probac(\Omega)$,
    \begin{align*}
      \Wass_2^2(\sigma,\mu_X) &= \max_{\psi\in \Rsp^N} \int_{\Omega} \min_i \nr{x - x_i}^2 - \psi_i \dd \sigma(x) + \sum_{1\leq i\leq N}
    \frac{\psi_i}{N}
      \\ &=
 \max_{\psi \in \Rsp^N} 
    \sum_{1\leq i\leq N} \int_{L_i(\psi)} \nr{x - x_i}^2 - \psi_i \dd
    \sigma(x) + \sum_{1\leq i\leq N}
    \frac{\psi_i}{N},
    \end{align*}
    thus giving
    $$
    \min_{\sigma \in K} \frac{1}{2\eps} \Wass_2^2(\sigma,\mu_X) =
    \min_{\sigma \in K} \max_{\psi \in \Rsp^N} \frac{1}{2\eps}
    \sum_{1\leq i\leq N} \int_{L_i(\psi)} \nr{x - x_i}^2 - \psi_i \dd
    \sigma(x) + \frac{1}{2\eps} \sum_{1\leq i\leq N}
    \frac{\psi_i}{N}
      $$ Switching the minimum and the maximum, we get
    the following dual problem 
    \begin{align*}
      \max_{\psi \in \Rsp^N} &\min_{\sigma \in L^1(\Omega), 0\leq \sigma \leq 1} \frac{1}{2\eps} \sum_{1\leq i\leq N} \int_{L_i(\psi)} \nr{x - x_i}^2 - \psi_i \dd \sigma(x) + \frac{1}{2\eps} \sum_{1\leq i\leq N} \frac{\psi_i}{N}
      = \max_{\psi\in\Rsp^N} D(\psi),
    \end{align*}
    where, we set
    $$B_i(\psi) := B(x_i, \sqrt{\psi_i}) = \{x \in \Rsp^d \mid \nr{x - x_i}^2 -
    \psi_i \leq 0 \},$$ and
    $$ D(\psi) = \frac{1}{2\eps} \sum_{1\leq i\leq N} \int_{L_i(\psi)
      \cap \Ball_i(\psi)} \nr{x - x_i}^2 - \psi_i \dd x -
    \frac{1}{2\eps} \sum_{1\leq i\leq N} \frac{\psi_i}{N}.$$ With
    similar arguments as in \cite[Theorem 1.1]
    {kitagawa2016convergence}, one can prove that $D$ is concave,
    $\mathcal{C}^1$, and that its partial derivatives are $$
    \partial_{\psi_i} D(\psi) = - \frac{1}{2\eps}\left(\abs{L_i(\psi)
    \cap \Ball_i(\psi)}- \frac{1}{N}\right).$$ It is easy to see that
    the maximum is attained in the dual problem, thus proving the
    existence of $\psi\leq 0$ satisfying \eqref{eq:crowd-my-dual}.
    Define $\sigma$ and $\phi$ as in the statement. Then, the property
    $\phi(1-\sigma)$ is obvious. In addition,
    $$ \phi(x) + \psi_i = \min(\min_j \nr{x - x_j}^2 - \psi_j,0)
    +\psi_i \leq \nr{x - x_i}^2,$$ so that the pair $(\phi,\psi)$ is
    admissible in the dual Kantorovich problem. It is also optimal in
    the optimal transport problem between $\sigma$ and $\mu_X$ by
    construction, since $\phi$ coincides with the $c$-transform of
    $\psi$ on the support of $\sigma$. This shows that
    \begin{align*}
      \frac{1}{2\eps} \dist^2_K(\mu_X) \leq \frac{1}{2\eps}
    \Wass_2^2(\sigma,\mu_X) &= \frac{1}{2\eps} \int \phi \dd\sigma -
    \frac{1}{2\eps} \int \psi \dd \mu_X \\
    &\leq \frac{1}{2\eps} \sum_i \int_{L_i(\psi) \cap \Omega\cap B_i(\psi)} \phi \dd x -
    \frac{1}{2\eps} \int \psi \dd \mu_X = D(\psi).
    \end{align*}
    Since the converse inequality holds by weak duality, we get strong
    duality, and in particular $\sigma$ is the solution to the primal
    problem.
    \end{proof}

  The following proposition, dealing with the linear diffusion case,
  is obtained in a very similar manner (one can for instance use
  Proposition~8.6 in \cite{otam} to get the optimality condition for the
  dual problem). We also refer the reader to Theorems 3.1--3.2 in
  \cite{bourne2018semi}, where similar results are shown for more
  general functionals.
    \begin{proposition} \label{prop:my-boltzmann}
    Let $H:\Prob(\Omega)\to\setR$ be Boltzmann's functional, and
    $X\in\setR^{Nd}\setminus D_N$. Then, there exists $\psi \in\Rsp^N$
    such that
    \begin{equation} \label{eq:ent-my-dual}
      \forall i,~ \int_{L_i(\psi)} e^{- \frac{1}{2\eps}(\nr{x - x_i}^2 - \psi_i)} \dd
      x = \frac{1}{N}
    \end{equation}
    Given such a $\psi$, define $\phi = \min_{i} \nr{\cdot - x_i}^2
    - \psi_i$ and $\sigma = e^{-\frac{\phi}{2\eps}}\mathbf{1}_\Omega$.
    Then,
    \begin{itemize}
    \item[(a)] $\sigma\in\Probac(\Omega)$ is the unique minimizer of
      $\min_{\Probac(\Omega)}\frac{1}{2\eps} \Wass_2^2(\cdot,\mu_X) + H(\cdot)$.
    \item[(b)] $\frac{1}{2\eps} \phi + \log(\sigma) = 0$
    \item[(c)] $(\phi,\psi)$ is a pair of optimal Kantorovich potentials in
      the transport from $\sigma$ to $\mu_X$.
      \end{itemize}
    \end{proposition}
    
    \begin{remark}
      In practice, equations \eqref{eq:crowd-my-dual} and
      \eqref{eq:ent-my-dual} are solved using the same damped Newton
      algorithm as in \cite{kitagawa2016convergence}. The cells
      $L_i(\psi)$ are computed using computational geometry
      techniques ensuring a near-linear computational time in
      $2$D. The integrals are computed exactly in the crowd motion
      case, and using quadratures ensuring a negligible numerical
      error in the linear diffusion case.
    \end{remark}

    \begin{remark}
      The computation of the Moreau-Yosida regularization of the
      congestion constraint and the entropy is implemented in the
      open-source library \texttt{sd-ot}, which is available at
      \texttt{https://github.com/sd-ot}. The numerical schemes for
      crowd motion and linear diffusion are implemented as examples in
      the Python package.
    \end{remark}

    \subsection{Numerical experiments (crowd motion)}
    In this paragraph, we consider $\Omega\subseteq \Rsp^2$ a compact
    domain, $V:\Omega\to\Rsp$ a potential, and we define as usual the
    congestion term $F: \Prob(\Omega)\to\Rsp$ by $F(\mu) = 0$ if
    $\mu$ has density $\leq 1$ and $+\infty$ if not. We consider the
    discretization of the crowd motion model explained above: an
    initial point set $X^0 = (x_1^0,\hdots,x_N^0)$ is evolved through
    the ODE system
    $$ \begin{cases}
      \frac{1}{N} \dot{x_i}(t) =
      -\nabla_{x_i} F_\eps(x_1(t),\hdots,x_{N_h}(t)) - \frac{1}{N} \nabla V(x_i(t)), \\
      x_i(0) = x_i^0 \end{cases}$$
    which we discretize using a simple explicit Euler scheme:
    $$ \frac{x_i^{k+1} - x_i^k}{\tau} = - \nabla_{x_i}
    F_\eps(x_1^k,\hdots,x_{N_h}^k) - \nabla V(x_i^k).$$
    Propositions~\ref{prop:derivative}--\ref{prop:derivative-crowd}
    can be used to compute the gradient of the regularized congestion
    term $F_\eps$. Figure~\ref{fig:crowd-gradient} illustrates this
    computation by showing a point set $X = (x_1,\hdots,x_N)$, the
    projected measure $\sigma \in \Probac(\Omega)$, $\sigma\leq 1$ and
    the gradient $(\nabla_{x_i} F_\eps(X))_{1\leq i\leq N})$.
    
 \begin{figure}\label{fig:converging-corridor}
   \centering
   \begin{tabular}{ccc}
   \includegraphics[height=.20\textwidth]{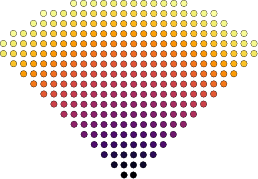}&
   \includegraphics[height=.18\textwidth]{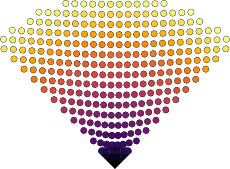}&
   \includegraphics[height=.17\textwidth]{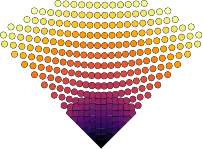}\\
   \includegraphics[height=.16\textwidth]{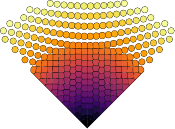}&
   \includegraphics[height=.13\textwidth]{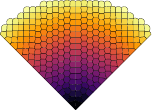}&
   \includegraphics[height=.13\textwidth]{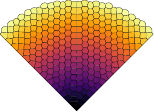}
   \end{tabular}
   \caption{Evolution of particles in the radial case, with
     $h=\frac{1}{40}$. The color of the cell $i$ is related to the
     position of the point $x_i^0$, allowing to visualize the movement
     of particle.}
 \end{figure}

    \paragraph{Radial case} As a first test case, we consider a
    simple problem with radial symmetry, introduced in
    \cite[Section~5]{gf-cm}, and whose solution is explicit.  The
    domain is the set $\Omega = \{ x\in \Rsp^2 \mid x_2 \geq
    \abs{x_1}, \nr{x}\leq R\}$, and the potential is given by $V(x) =
    \nr{x}$. In our experiment, we assume that $R = 2$ and $\alpha =
    \frac{1}{\pi}$, so that $\rho^0 = \alpha \mathbf{1}_\Omega$ is a
    probability measure.  As shown in \cite{gf-cm}, the evolution of
    the crowd is then given by
    $$ \rho_t(x) = \begin{cases} 1 & \hbox{ if } r\in [0,b(t)[ \\
          \alpha\left(1+\frac{t}{\nr{x}}\right) & \hbox{ if } r\in [b(t),R-r[ \\
              0& \hbox{ if } r\in [R-t,T], \end{cases}$$
    where $b$ is a solution of
    $$ \begin{cases} b(0) = 0 \\ b'(t) = \alpha \frac{b(t) + t}{b(t) - \alpha (b(t)+t)}.
      \end{cases}$$
 Given $h > 0$, we denote $N_h = \Card(\Omega\cap h \Zsp^2)$ and we
 let $x_1^0,\hdots, x_{N_h}^0$ be the an arbitrary numbering of the
 points in the intersection $\Omega\cap h \Zsp^2$.  In all
 experiments, we set $\tau = \frac{h}{2},$ $\eps = h$ and
 $T=1$. Figure~\ref{fig:converging-corridor} displays the evolution of
 the Laguerre cells at six time steps. To get error estimates, we
 measure the Wasserstein distance between:
 \begin{itemize}
 \item $\bar{\rho}_t = (x \mapsto \nr{x})_{\#} \rho_t \in
   \Prob(\Rsp)$, which is the distribution of distances from the
   origin, computed from the exact solution $\rho_t$;
 \item $\bar{\mu}^k = \frac{1}{N_h} \sum_{1\leq i\leq N_h}
   \delta_{\nr{x_i^k}} \in \Prob(\Rsp)$ the distribution of distances
   from the origin, computed on the discrete solution $\mu^k =
   \frac{1}{N_h}\sum_{1\leq i\leq N_h} \delta_{x_i}$.
 \end{itemize}
 The relation between $h$ and $\mathrm{err}_h = \max_{0\leq k \leq
   \frac{T}{\tau}} \Wass_2(\bar{\rho}_{k\tau}, \bar{\mu}^k),$ as
 reported in Table~\ref{table:errors}, suggests a near-linear
 convergence rate.

 \begin{table}
   \centering
  \begin{tabular}{|c|c|c|c|c|c|c|}
    \hline
    $h$ & $\frac{1}{20}$ & $\frac{1}{30}$ & $\frac{}{40}$& $\frac{1}{50}$& $\frac{1}{100}$& $\frac{1}{200}$\\
    \hline
    $\mathrm{err}_h$ & $5.24\cdot 10^{-2}$  & $3.06\cdot 10^{-2}$ & $2.15\cdot 10^{-2}$ & $1.70\cdot 10^{-2}$ & $4.96\cdot 10^{-3}$ & $2.80 \cdot 10^{-3}$\\
    \hline
  \end{tabular}  
   \caption{\label{table:errors} Error $\mathrm{err}_h$ between the exact and numeric solution to
     the crowd motion model as a function of the space-discretization
     $h$.}
\end{table}

 \paragraph{Bimodal case} In this case,
  is obtained as the union of three squares $\Omega = \Omega_\ell \cup
  \Omega_r\cup \Omega_c$: two ``rooms'' $\Omega_\ell$ and $\Omega_r$
  joined by a corridor $\Omega_c$, where
  $$\Omega_\ell = [0,\alpha]^2, \quad \Omega_r =
  [\frac{4}{3}\alpha,\frac{7}{3}\alpha] \times [0,\alpha], \quad
  \Omega_c =
        [\alpha,\frac{4}{3}\alpha]\times[\frac{1}{3}\alpha,\frac{2}{3}\alpha],
        \quad \alpha = \frac{2}{\sqrt{\pi}}$$ The crowd is initially
        located in the left room $\Omega_\ell$ and the potential $V$
        is constructed as the distance function to the two corners
        $\{(\frac{7}{3}\alpha,\alpha), (\frac73\alpha,0)\}$ of the
        right room $\Omega_r$. More precisely, $V$ is obtained as the
        solution to the following Eikonal equation, which is computed
        using a fast marching method:
 $$ \begin{cases}
   \nr{\nabla V} = 1, \\
   V(\frac{7}{3}\alpha,\alpha) = V(\frac{7}{3}\alpha,0) = 0.
 \end{cases} $$
 Given $h > 0$, we denote $N_h = \Card(\Omega_\ell\cap h \Zsp^2)$ and
 we let $x_1^0,\hdots, x_{N_h}^0$ be the list of points in
 $\Omega_\ell\cap h \Zsp^2$, so that the crowd is initially located on
 the left square $\Omega_\ell$. Here, we set the final time to $T=3$,
 and as before, we have $\eps=h$, $\tau=\frac{h}{2}$.
 \begin{figure}
   \centering
   \resizebox{.9\textwidth}{!}{
     \begin{tabular}{ccc}
   \includegraphics[width=.32\textwidth]{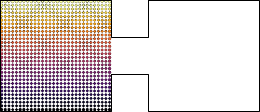} &
   \includegraphics[width=.32\textwidth]{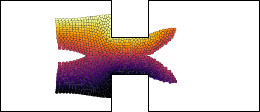}&
   \includegraphics[width=.32\textwidth]{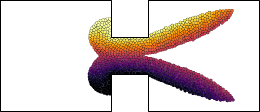}\\
   \includegraphics[width=.32\textwidth]{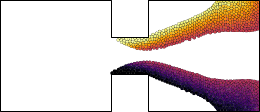}&
   \includegraphics[width=.32\textwidth]{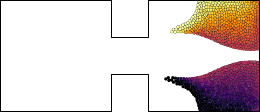}&
   \includegraphics[width=.32\textwidth]{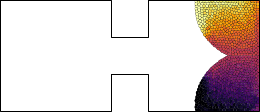}
   \end{tabular}}
   \caption{The distribution of the crowd computed at 6 different
     timesteps, with $h = \frac{\alpha}{30}$. The color of the
     Laguerre cells encodes the value of the $y$ coordinate of the
     corresponding particle at $t=0$.\label{fig:bimodal-low}}
    \end{figure}

 \begin{figure}
      \centering
      \resizebox{.93\textwidth}{!}{
             \begin{tabular}{ccc}
   \includegraphics[width=.32\textwidth]{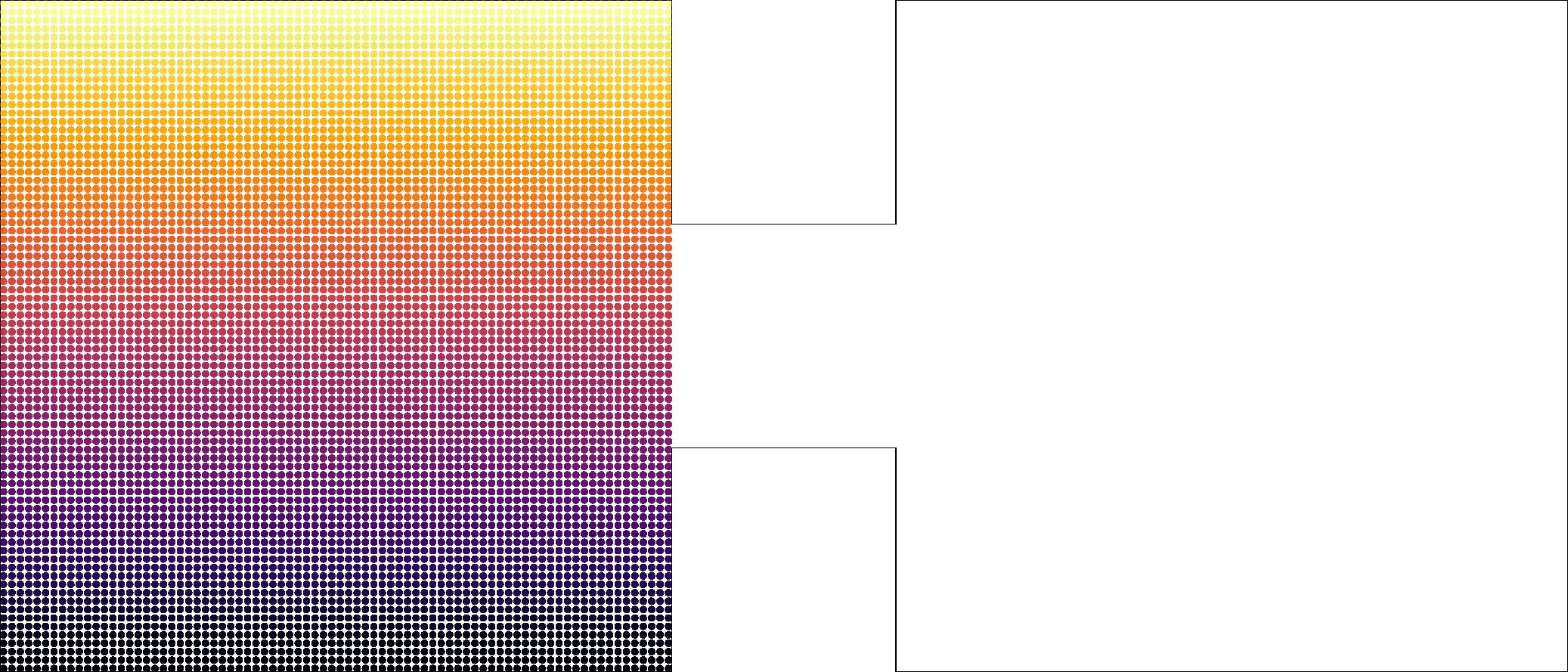}&
   \includegraphics[width=.32\textwidth]{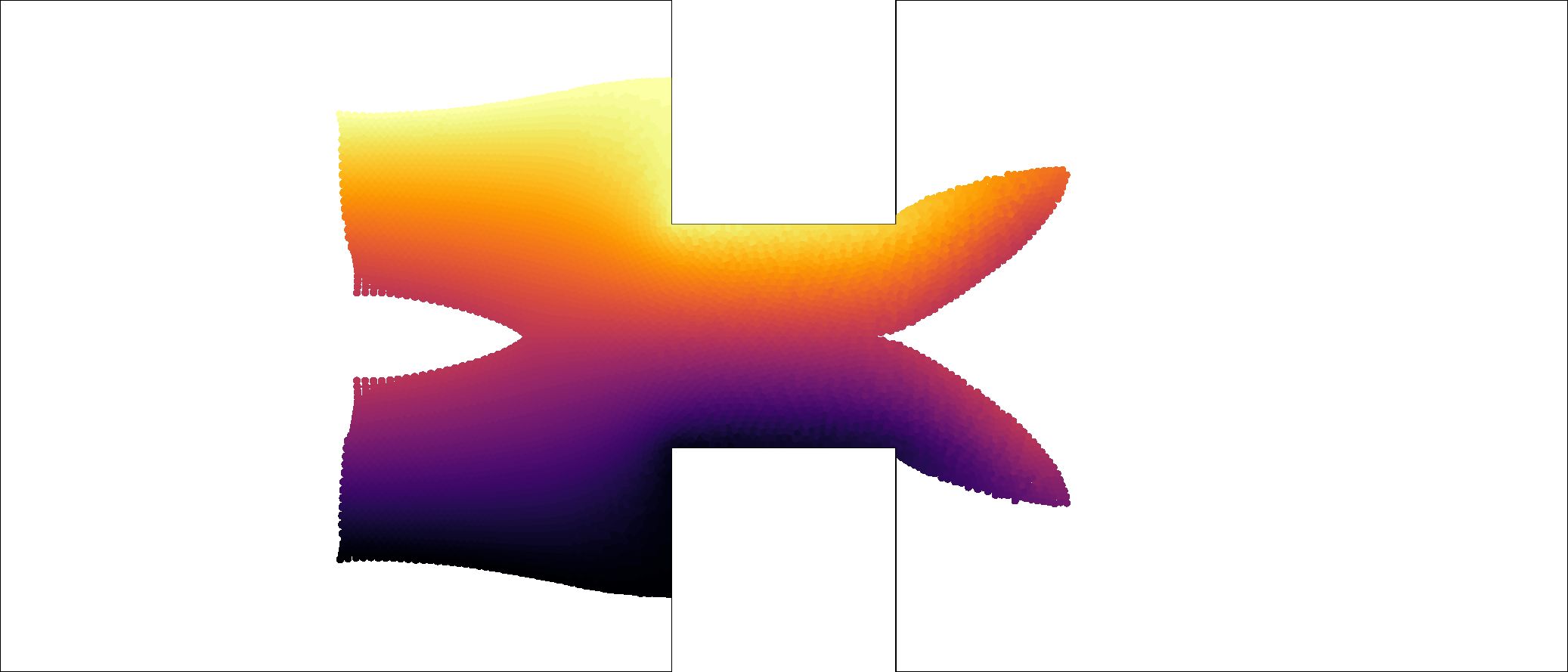}&
   \includegraphics[width=.32\textwidth]{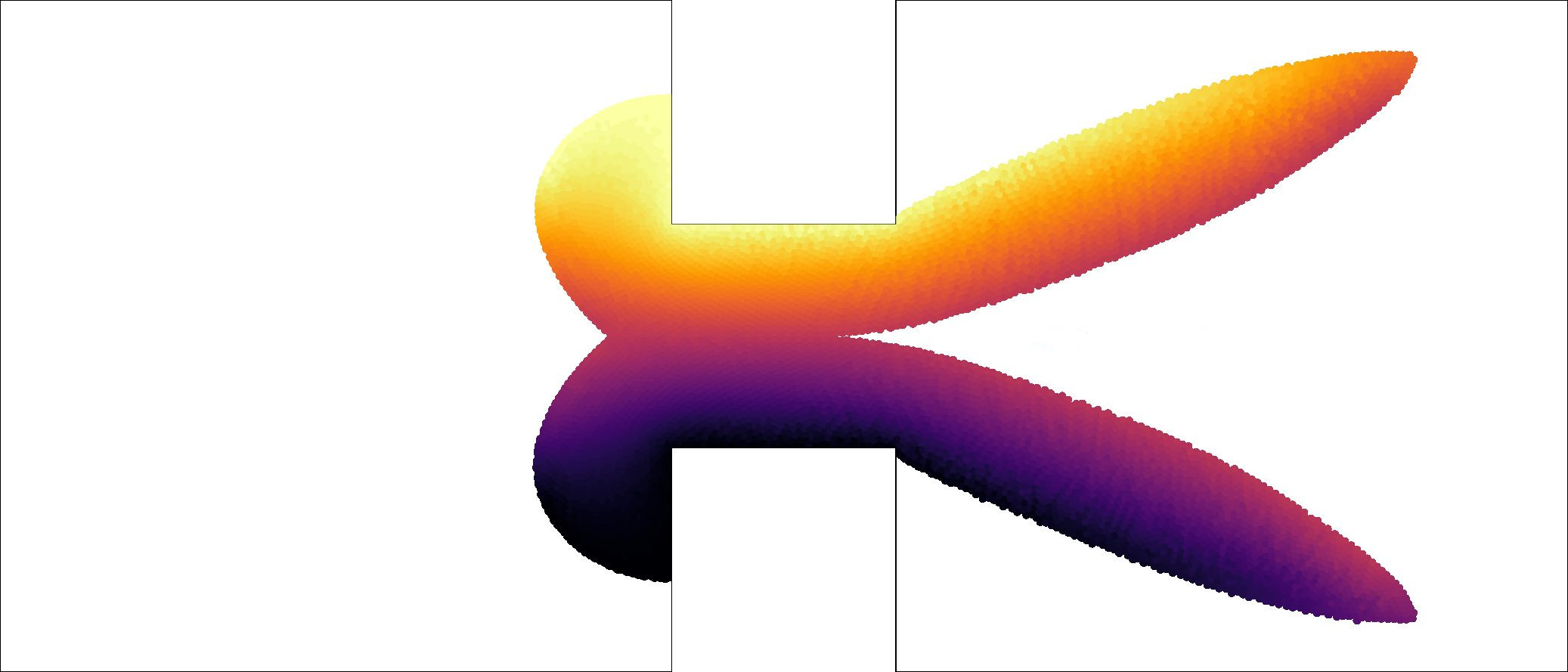}\\
   \includegraphics[width=.32\textwidth]{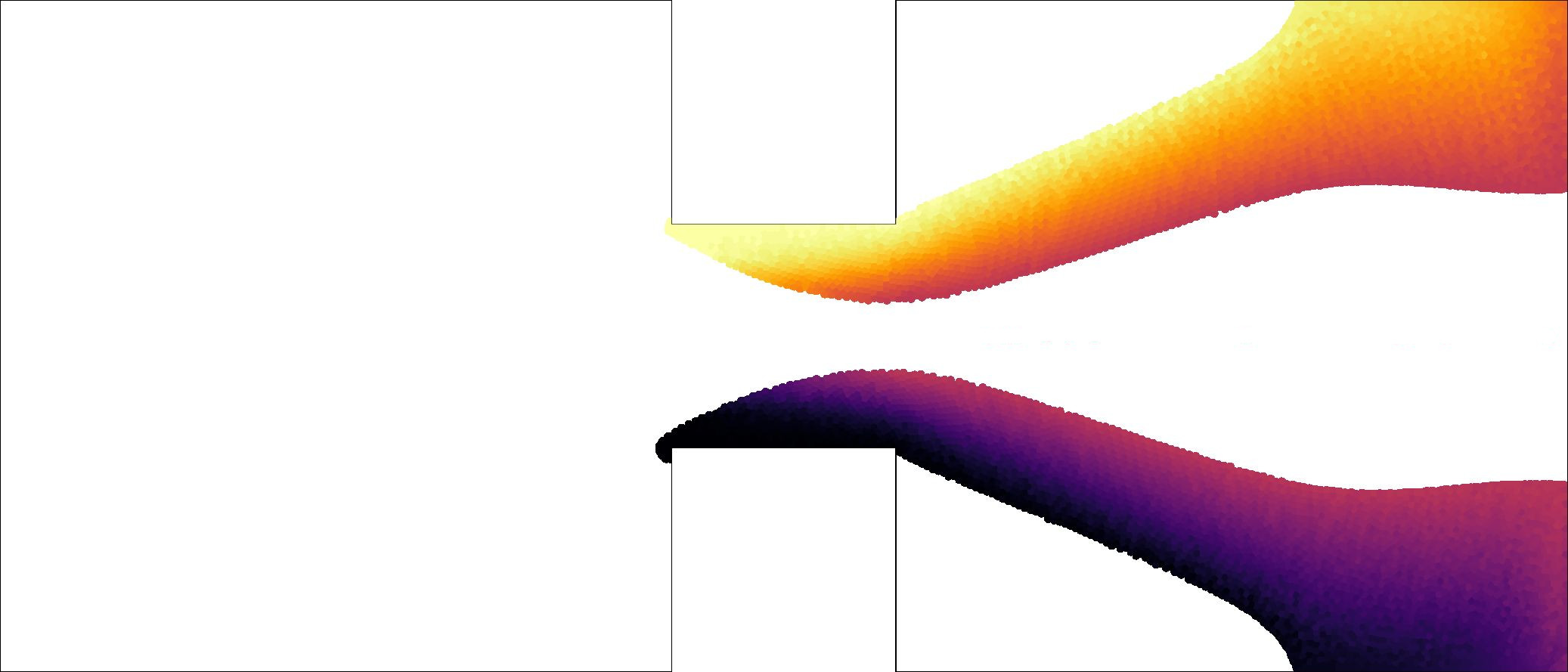}&
   \includegraphics[width=.32\textwidth]{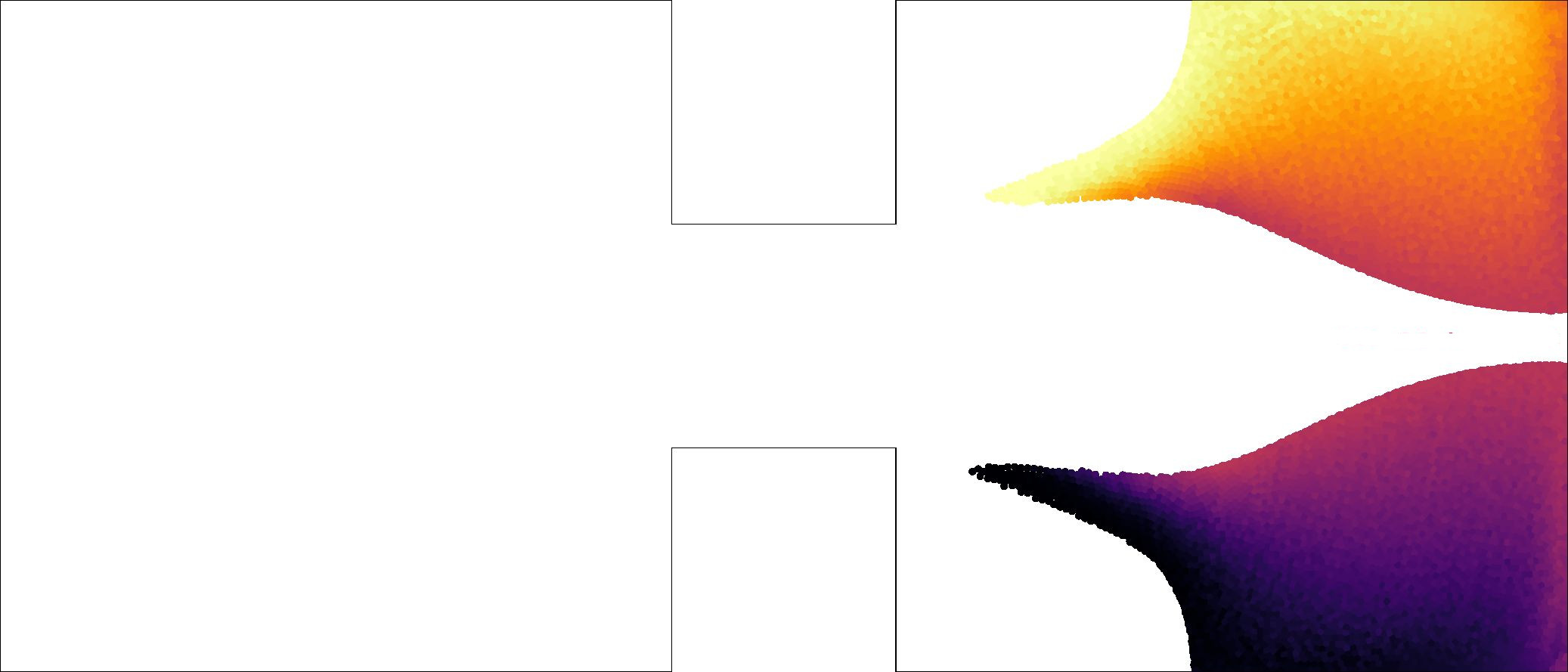}&
   \includegraphics[width=.32\textwidth]{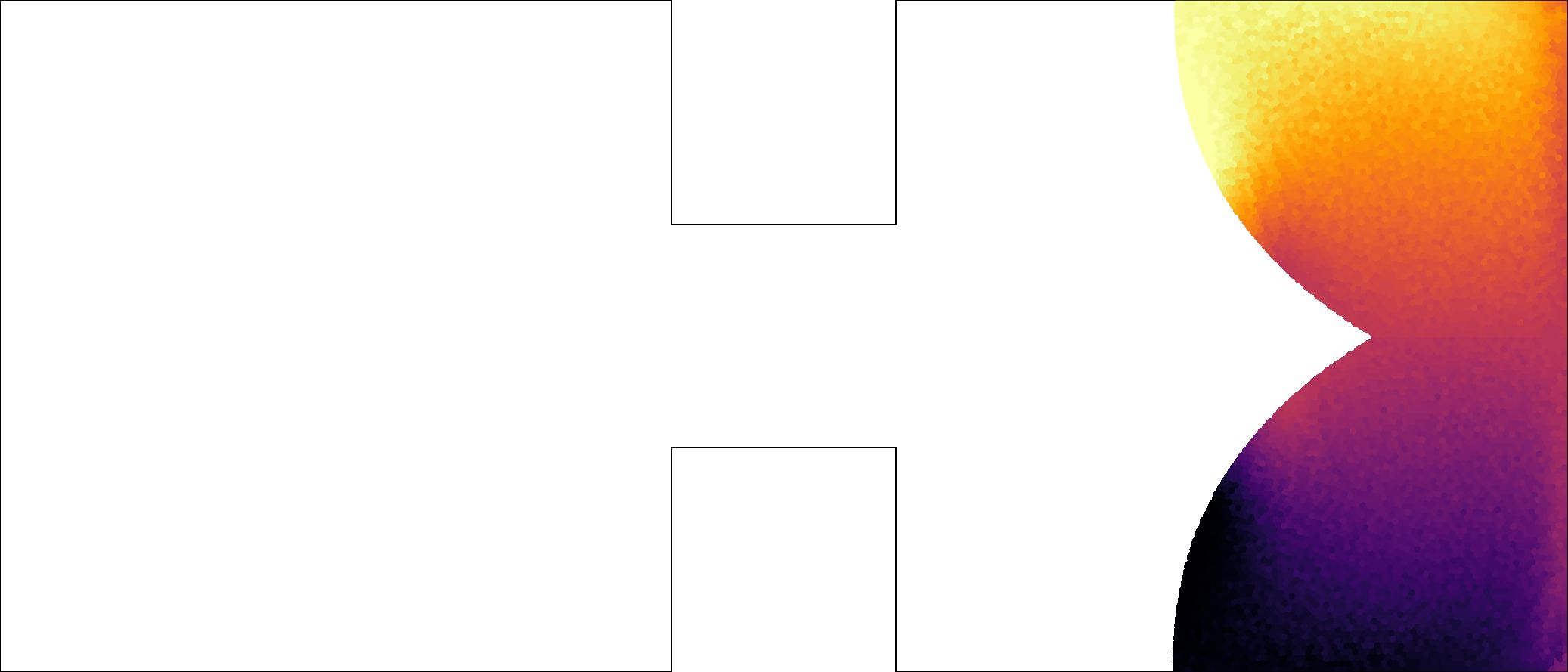}\end{tabular}}
   \caption{The distribution of the crowd computed at 6 different
     timesteps, with $h = \frac{\alpha}{80}$.  \label{fig:bimodal-high}}
 \end{figure}
 
 \begin{figure}
   \begin{center}
     \centering
     \resizebox{.93\textwidth}{!}{
       \begin{tabular}{cc}
         \includegraphics[width=.4\textwidth]{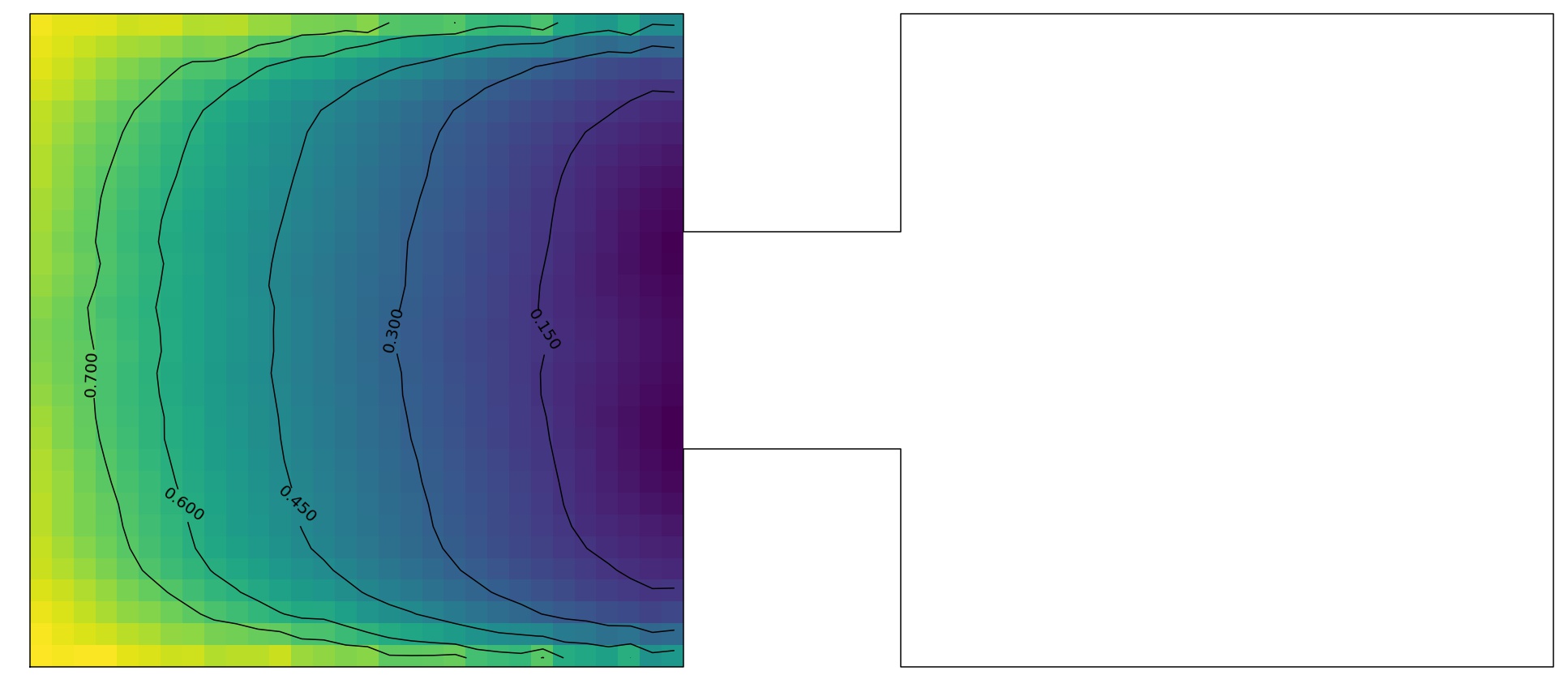}&
         \includegraphics[width=.4\textwidth]{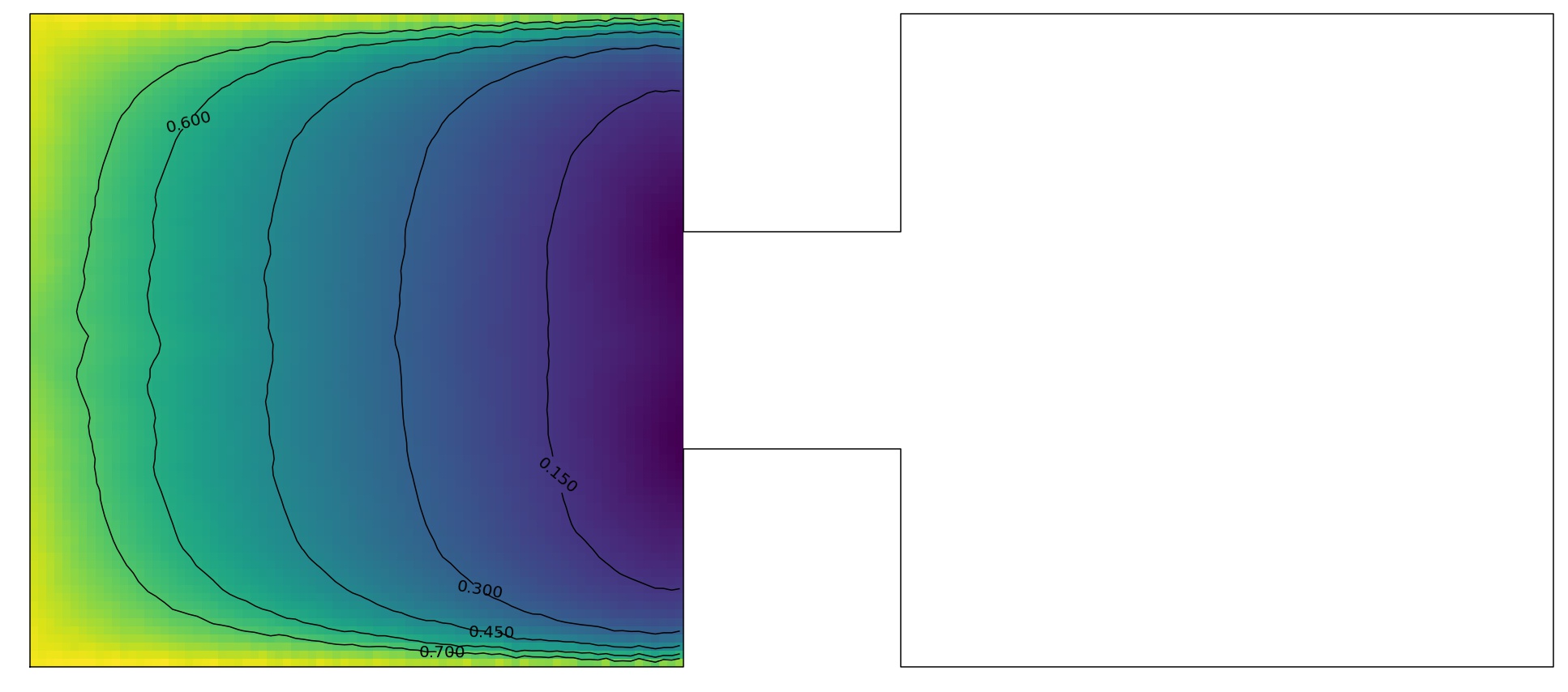} \\
   \includegraphics[width=.4\textwidth]{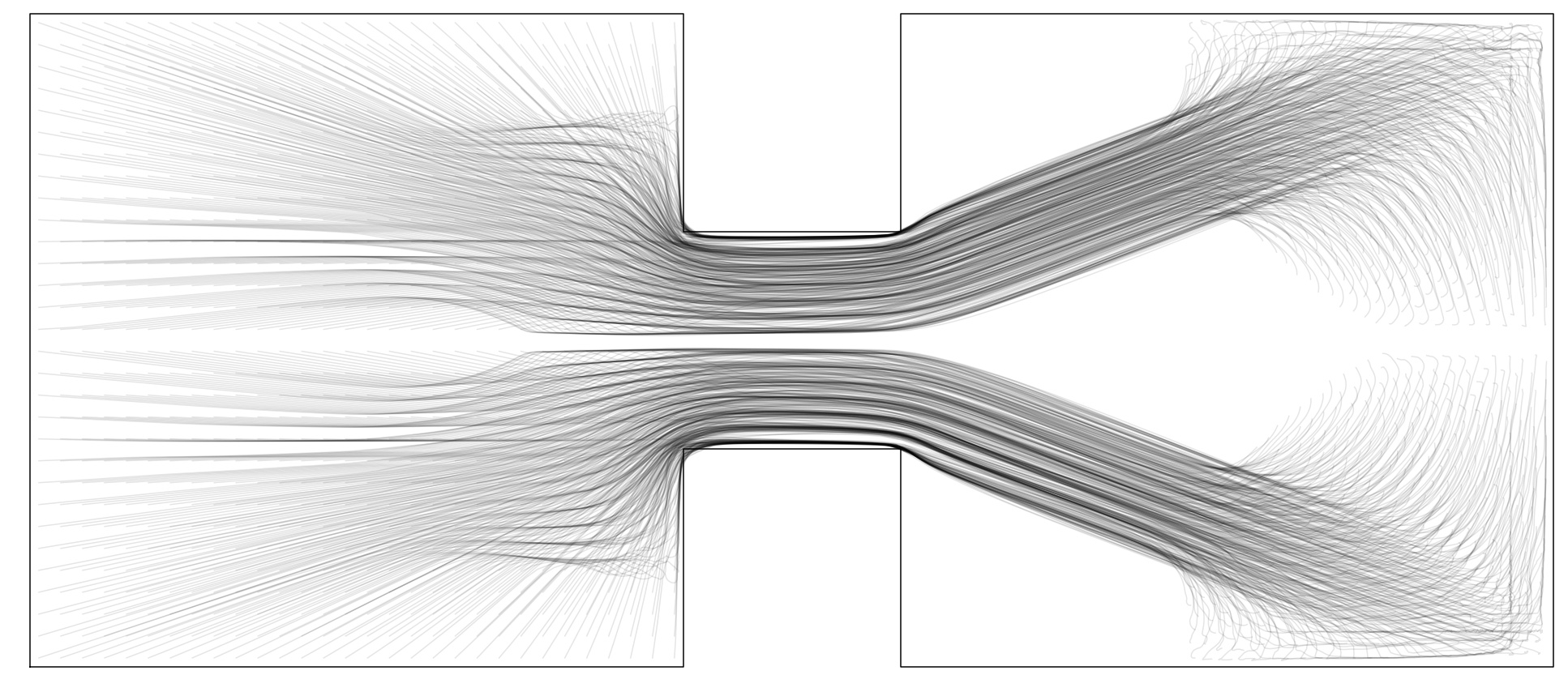}&
   \includegraphics[width=.4\textwidth]{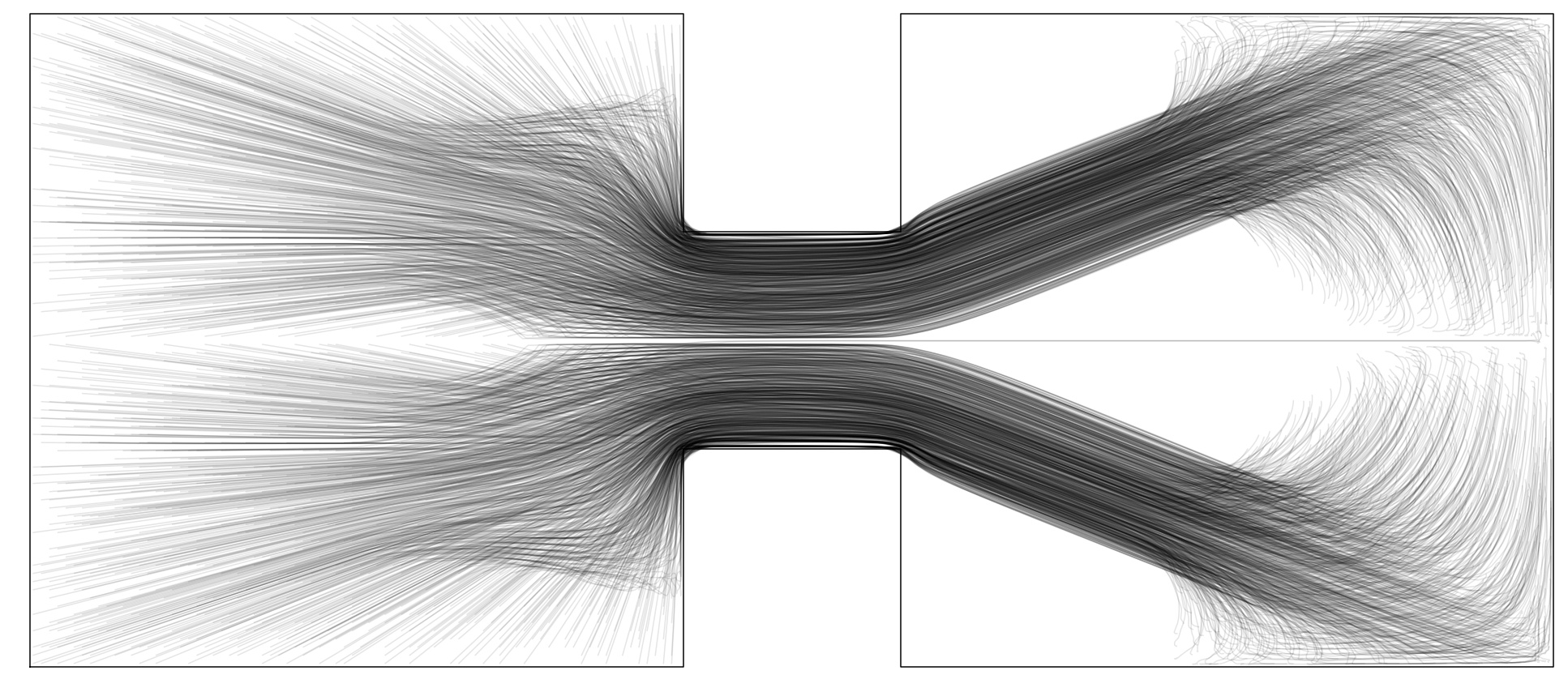} \\
   \includegraphics[width=.4\textwidth]{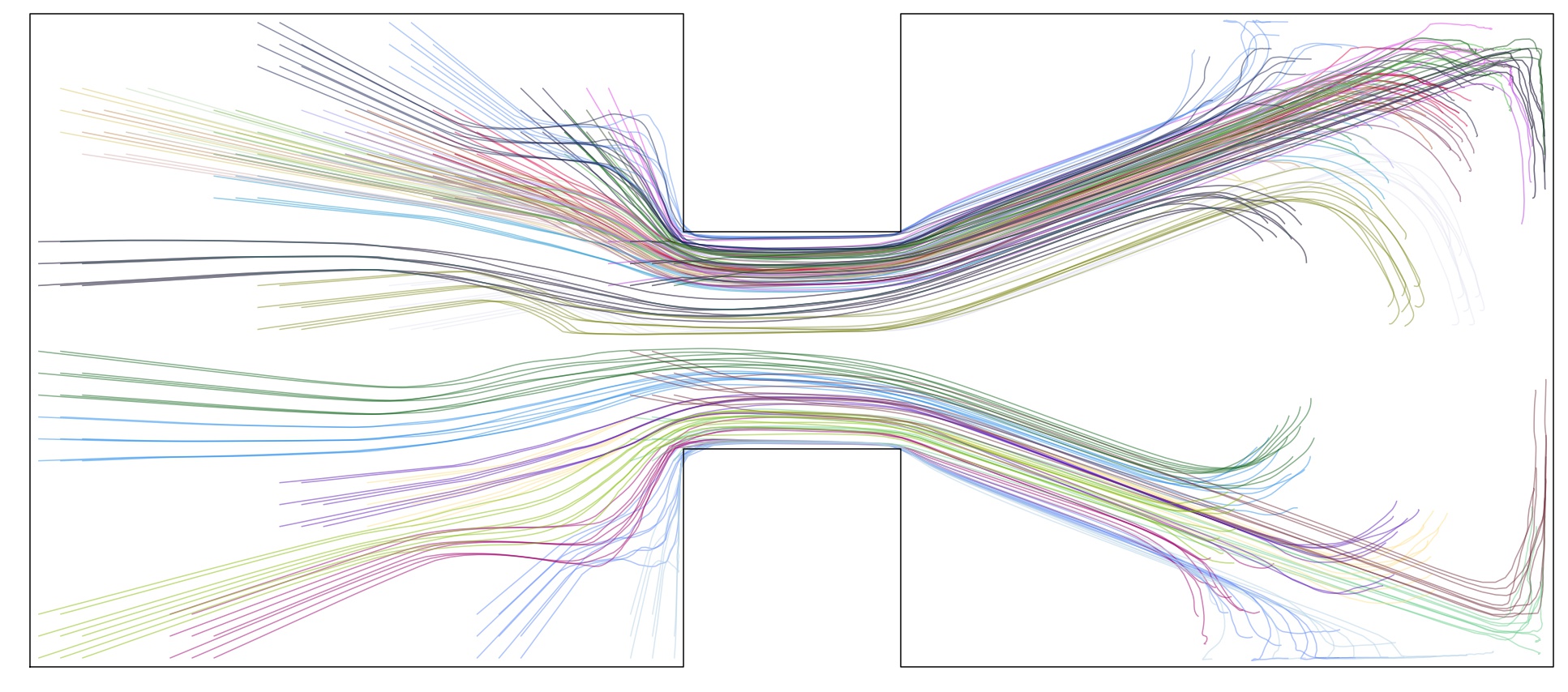}&
   \includegraphics[width=.4\textwidth]{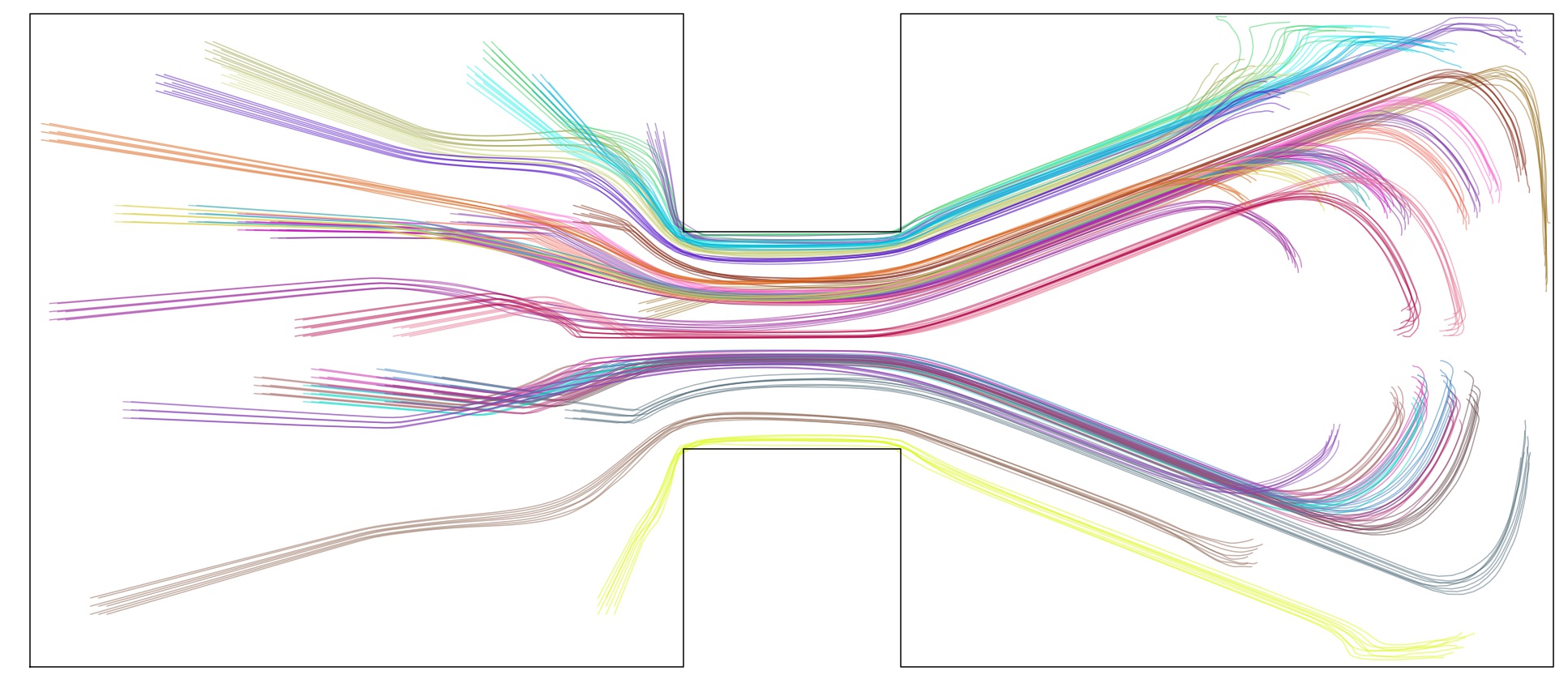}
   \end{tabular}}
   \caption{The left (resp. right) column corresponds to $h =
     \frac{\alpha}{30}$ (resp. $h = \frac{\alpha}{80}$). The first row
     display the timeout function defined in \eqref{eq:timeout}, which
     measures the time taken by a particle to leave the corridor. The
     second row displays the trajectories of all particles. The third
     row shows the trajectories of 20 randomly chosen cluster of
     particles.
     \label{fig:timeout-trajectories}}
   \end{center}
 \end{figure}
 
 \paragraph{Visualization} In Figures  \ref{fig:bimodal-low} and  \ref{fig:bimodal-high},
 we visualize the distribution of the crowd at the timesteps $t_i =
 \frac{i/5} T$ for $0\leq i\leq 5$, for two space discretizations $h =
 \frac{\alpha}{30}$ and $h=\frac{\alpha}{80}$. In
 Figure~\ref{fig:timeout-trajectories}, we highlight some features of
 the Lagrangian trajectory. First, given a particle $x_i^0$, we can
 compute the minimum time required for the particle to enter the right room: 
 \begin{equation} \label{eq:timeout}
   \tau_i := \tau \min \{ k \in \mathbb{N} \mid x_i^k \in \Omega_r \}.
 \end{equation}
 The exit time $\tau_i$ is displayed as a function of the particle
 coordinate $x_i^0$ at time $t=0$ in the first row of
 Figure~\ref{fig:timeout-trajectories}. This figure shows (as one
 could expect) that the exit time is not proportional to the distance
 to the ``door'' $\Omega_\ell \cap \Omega_c$, as people in front 
 tend to escape faster than those on side of the door. Finally,  the
 next two rows of \label{fig:timeout-trajectories} show the trajectory
 of the particles. The trajectories seem to be regular in time, but
 also seem to depend continuously on the initial condition (except
 near the non-differentiability locus of the potential $V$).

 \begin{remark}[On the assumption \eqref{eq:cm-boundpres}] In the 2D cases treated here,
   we cannot guarantee that our numerical solutions converge to a
   solution of the crowd motion equation since convergence requires
   the estimate \eqref{eq:cm-boundpres}. However, it is quite easy to
   see that this estimate holds if one is able to prove that the
   diameter of the Laguerre cells is bounded uniformly by $C
   (1/N)^{1/d}$ -- this is what is established in 1D in
   Section~\ref{sec:bounds}. Figure~\ref{fig:bimodal-low} (and in fact all our
   simulations) suggest that such an estimate is satisfied in practice.
 \end{remark}
 
 \begin{remark}[Lagrangian interpretation] \label{rem:crowdlag}
   The Eulerian crowd-motion equation \eqref{eq:cm-pde-press} can be
   turned into a Lagrangian equation by introducing the map $s_t \in
   \mathrm{L}^2(\rho_0,\Omega)$, which describes the displacement of
   the crowd from its position at time $t=0$ (more precisely, $s_t(x)$
   is the position at time $t$ which was at $x$ at time
   $0$). Formally, $s$ should satisfy the following system
   \begin{equation}\label{eq:crowd-lag}
     \begin{cases}
       \dot{s} = v\circ s \\
       s_0 = \mathrm{id} \\
       \rho = s_{\#} \rho_0 \\
   v = -\nabla p - \nabla V \\
   \rho \leq 1, p\geq 0, p(1-\rho) = 0,
     \end{cases}
   \end{equation}
and we expect that the numerical solution, shown in
Figure~\ref{fig:timeout-trajectories}, provides a piecewise-constant
(in space) approximation to $s$. We note however that the system
\eqref{eq:crowd-lag} has not been studied; showing existence of
solutions to this system would require to better understand the
regularity of the pressure $p$ appearing in \eqref{eq:cm-pde-press}.
 \end{remark}

 \subsection{Numerical experiments (diffusion)}
 In this paragraph, we consider $\Omega\subseteq \Rsp^2$ a compact
 domain, and we let $F$ be Boltzmann's functional. We consider the
 discretization of the heat equationl explained above: an initial
 point set $X^0 = (x_1^0,\hdots,x_N^0)$ is evolved through the ODE
 system \eqref{eq:adv-diff-ode} (with $V=0$), which we discretize
 again using a simple explicit Euler scheme:
 $$ \frac{x_i^{k+1} - x_i^k}{\tau} = - \nabla_{x_i}
 F_\eps(x_1^k,\hdots,x_{N_h}^k). $$

 In the numerical example presented in
 Figures~\ref{fig:heat30},\ref{fig:heat80} and \ref{fig:heatlag}, the
 initial density is uniform over a disk $D$, and approximated by the
 uniform measure over $h\Zsp^2 \cap D$. Despite the lack of
 convergence result in 2D, one can observe the consistency between the
 simulations with $h=\frac{1}{30}$ and $h=\frac{1}{80}$. Some of the
 cells near the boundary of the disk $D$ are very elongated; however
 this does not a priori prevent Assumption \eqref{eq:assump-diff} to
 hold, since there are few elongated cells and the assumed bound is on
 a mean quantity.

 \begin{remark}[Lagrangian interpretation] \label{rem:lagheat}
   The trajectories we construct (displayed in
   Figure~\ref{fig:heatlag}) should not be interpreted as realizations
   of solutions of the stochastic ODE associated with the heat
   equation. As in remark~\ref{rem:crowdlag}, we expect (but do not
   prove) that our numerical solutions actually approximate the
   solution to a Lagrangian equation which can be derived from the
   heat equation, namely
   \begin{equation}\label{eq:heat-lag}
     \begin{cases}
       \dot{s} = v\circ s \\
       s_0 = \mathrm{id} \\
   \rho = s_{\#} \rho_0 \\
   v = - \nabla \log\rho. \\
     \end{cases}
   \end{equation}
   In contrast with Remark~\ref{rem:crowdlag}, the existence of
   solutions to \eqref{eq:heat-lag} has been established in an article
   of Evans, Gangbo and Savin \cite{evans2005diffeomorphisms},
   assuming that the initial density $\rho_0$ is bounded from above
   and below. Their result can also be extended to some non-linear
   diffusion equations, under assumptions on the nonlinearity. This
   Lagrangian point of view has already been used to construct
   numerical schemes for nonlinear diffusion equations, see
   \cite{junge2017fully}.
\end{remark}
 
 \begin{figure}
   \begin{center}
   \includegraphics[width=.32\textwidth]{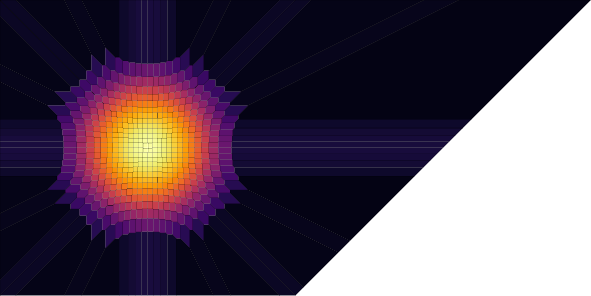}
   \includegraphics[width=.32\textwidth]{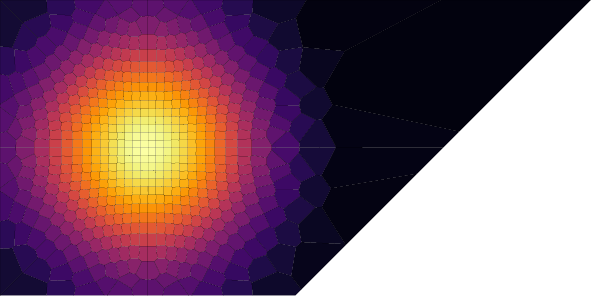}
   \includegraphics[width=.32\textwidth]{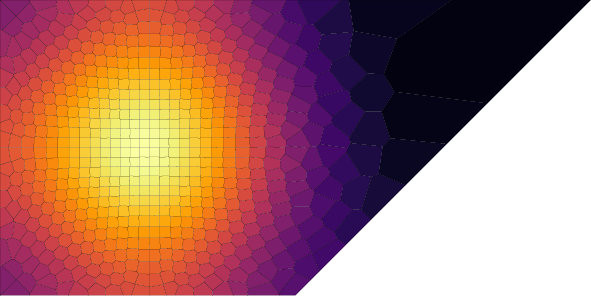}\\
   \includegraphics[width=.32\textwidth]{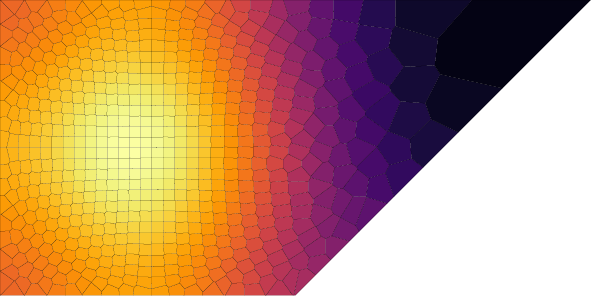}
   \includegraphics[width=.32\textwidth]{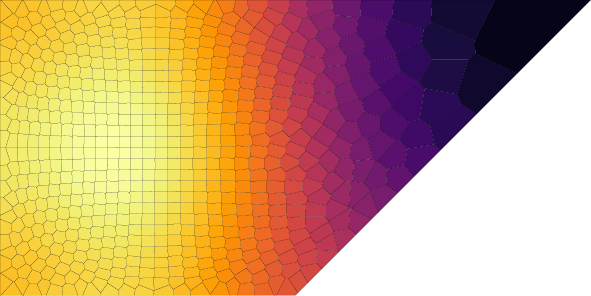}
   \includegraphics[width=.32\textwidth]{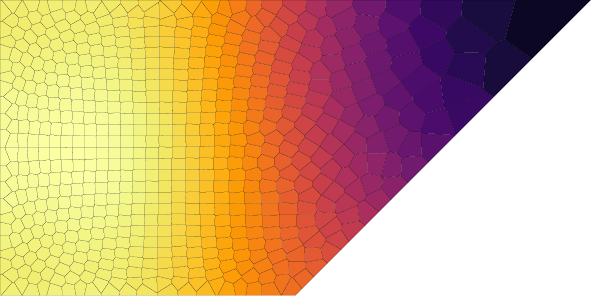}
   \end{center}
   \caption{Simulation of the heat equation, at several timesteps, for
     $h=\frac{1}{30}$. The color of the cell $L_i$ is proportional to
     $\exp(-\frac{1}{2\eps}(\nr{\beta_i(X) - x_i}^2 + \psi_i))$ (see
     \eqref{eq:ent-my-dual}). For better visibility, the density is
     represented on a color scale where the current maximum is
     always labeled with the same color (yellow). \label{fig:heat30}}
 \end{figure}

 \begin{figure}
      \begin{center}
   \includegraphics[width=.32\textwidth]{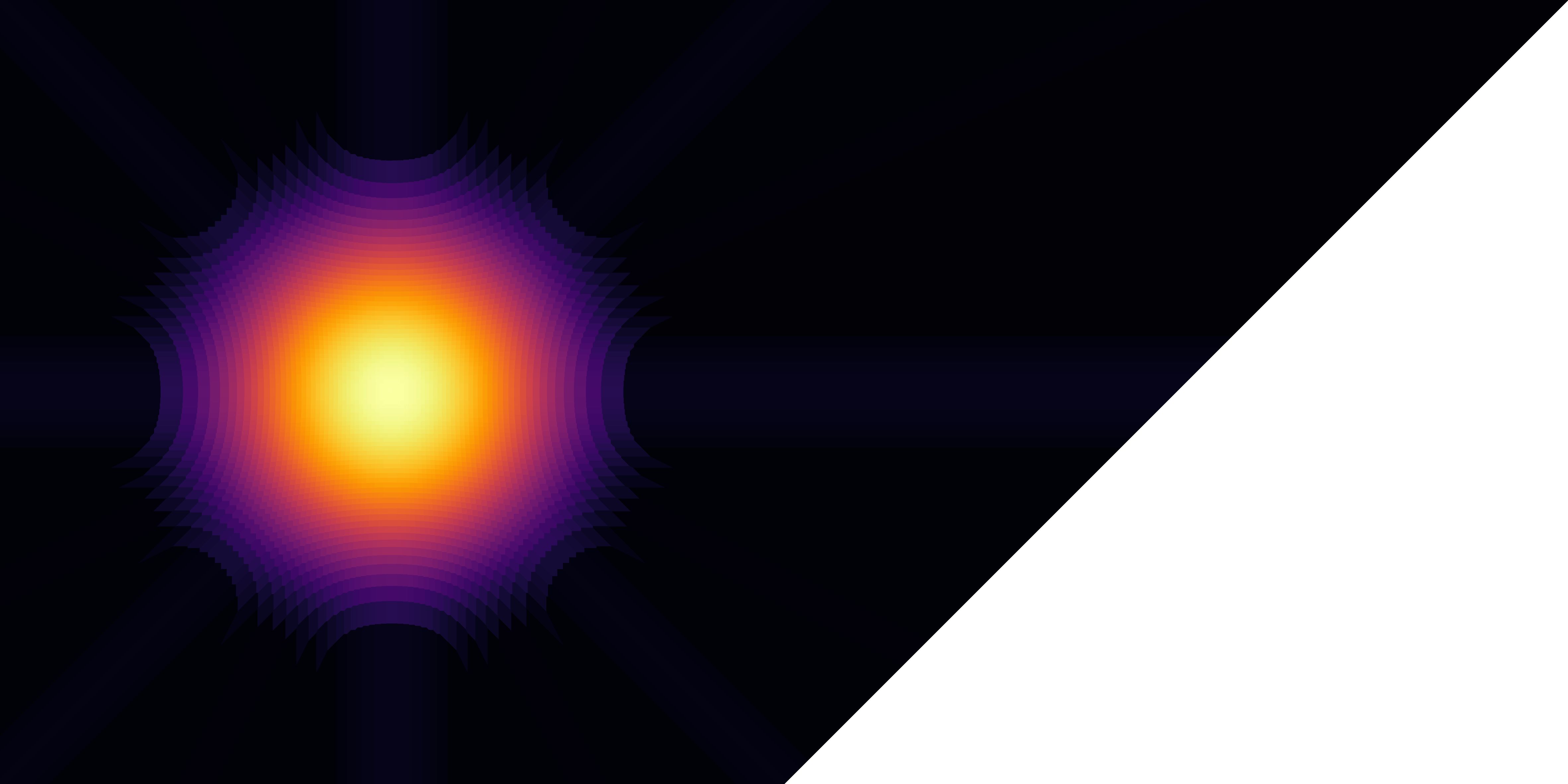}
   \includegraphics[width=.32\textwidth]{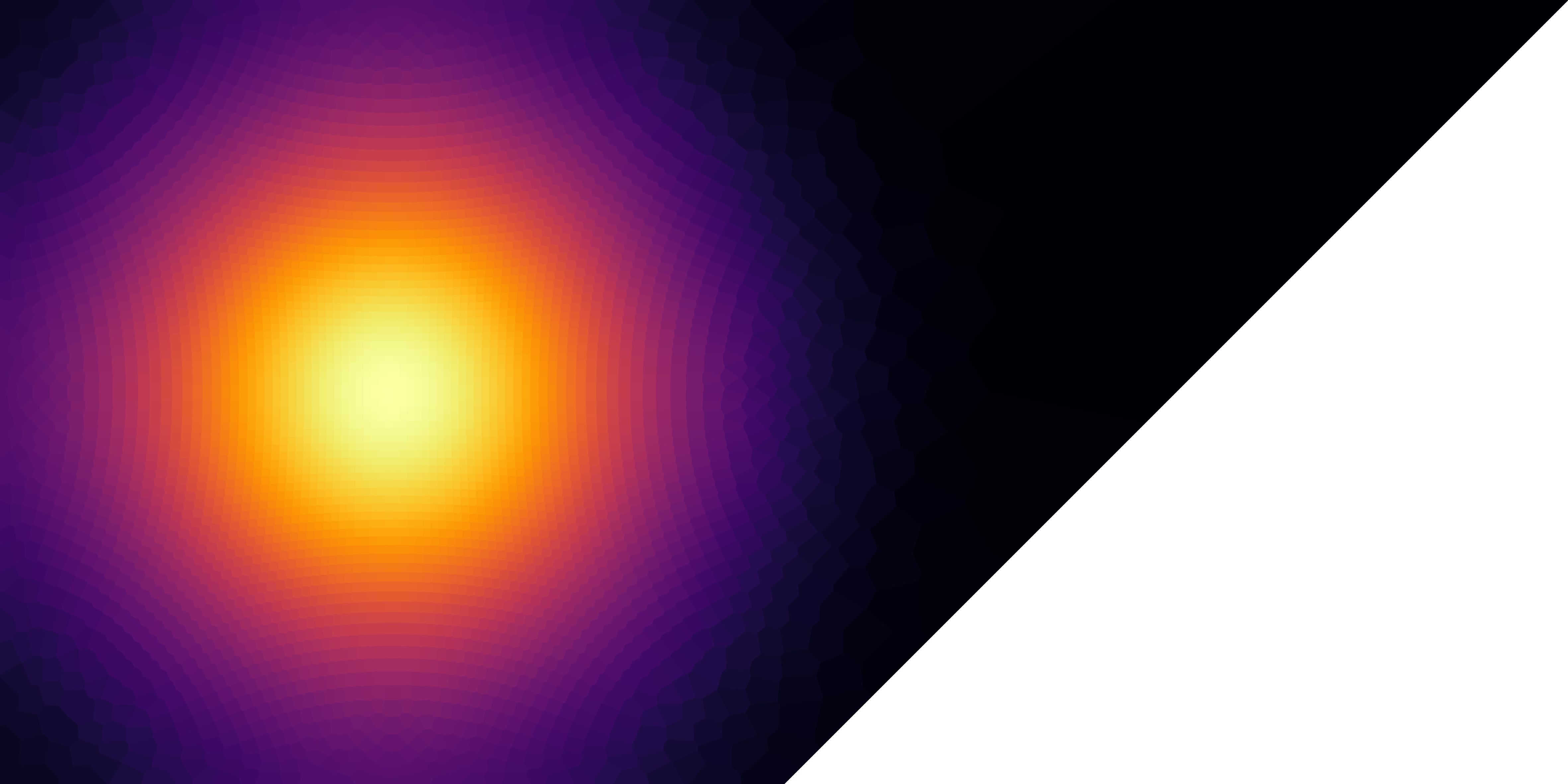}
   \includegraphics[width=.32\textwidth]{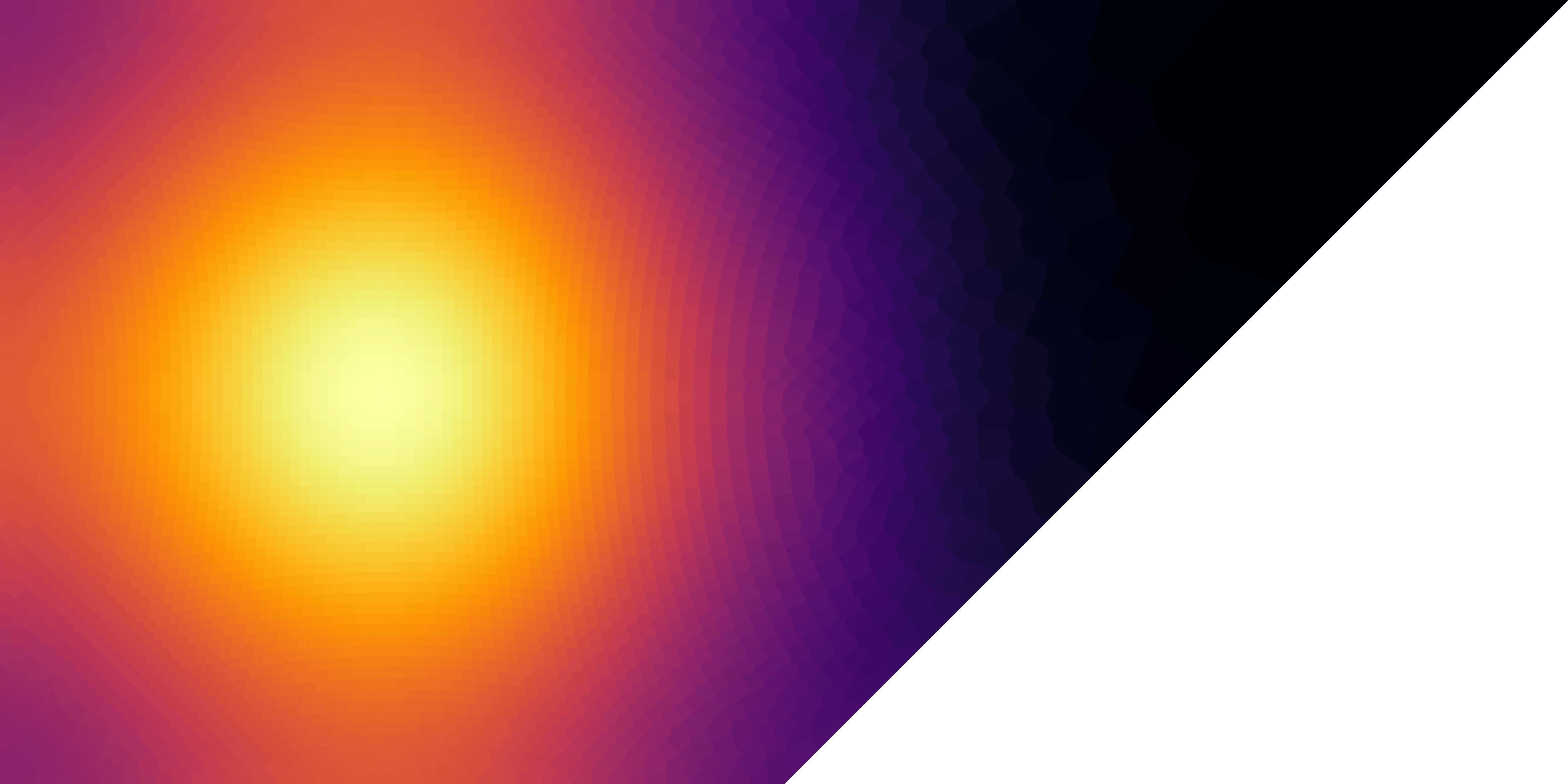}\\
   \includegraphics[width=.32\textwidth]{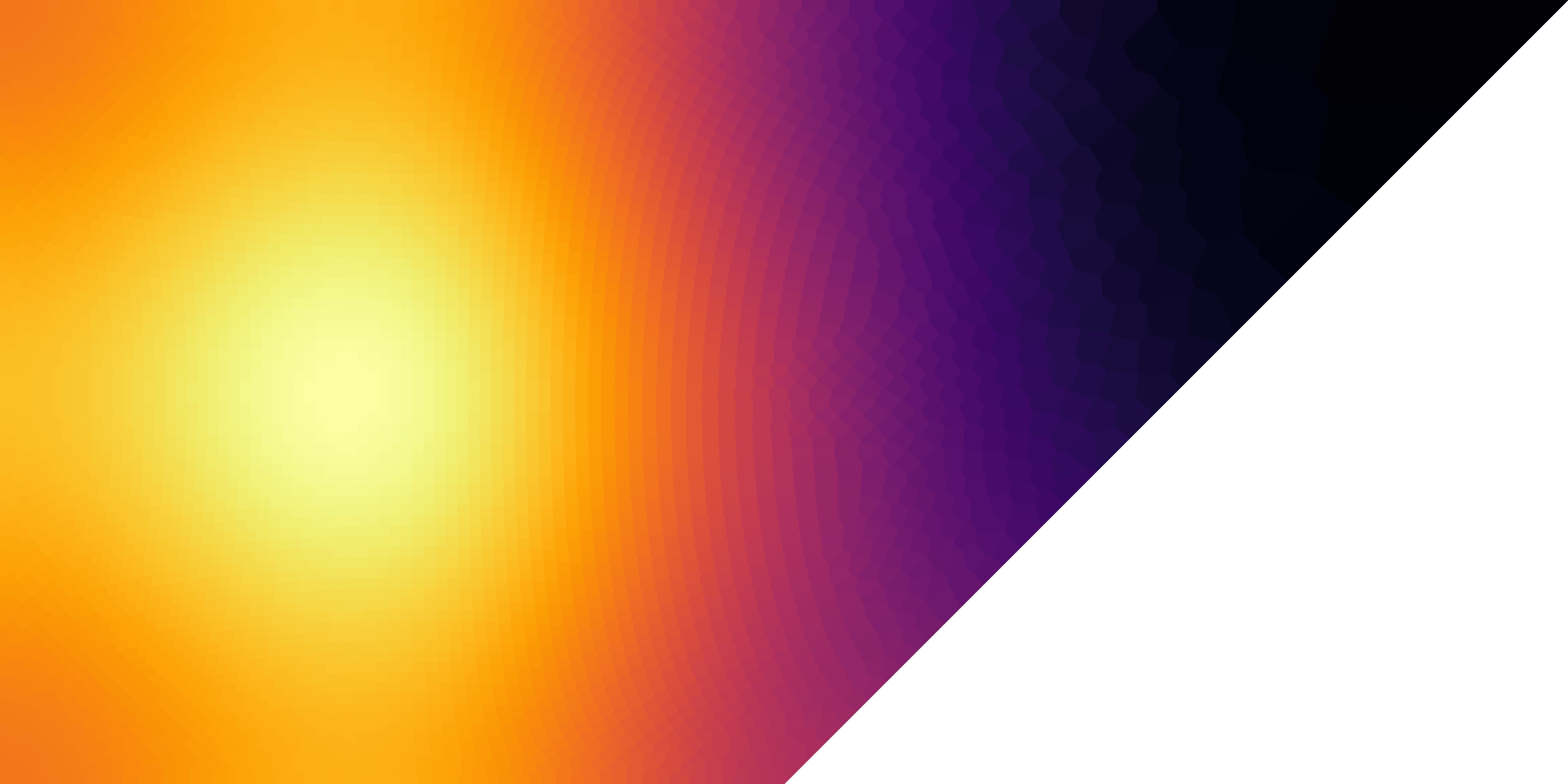}
   \includegraphics[width=.32\textwidth]{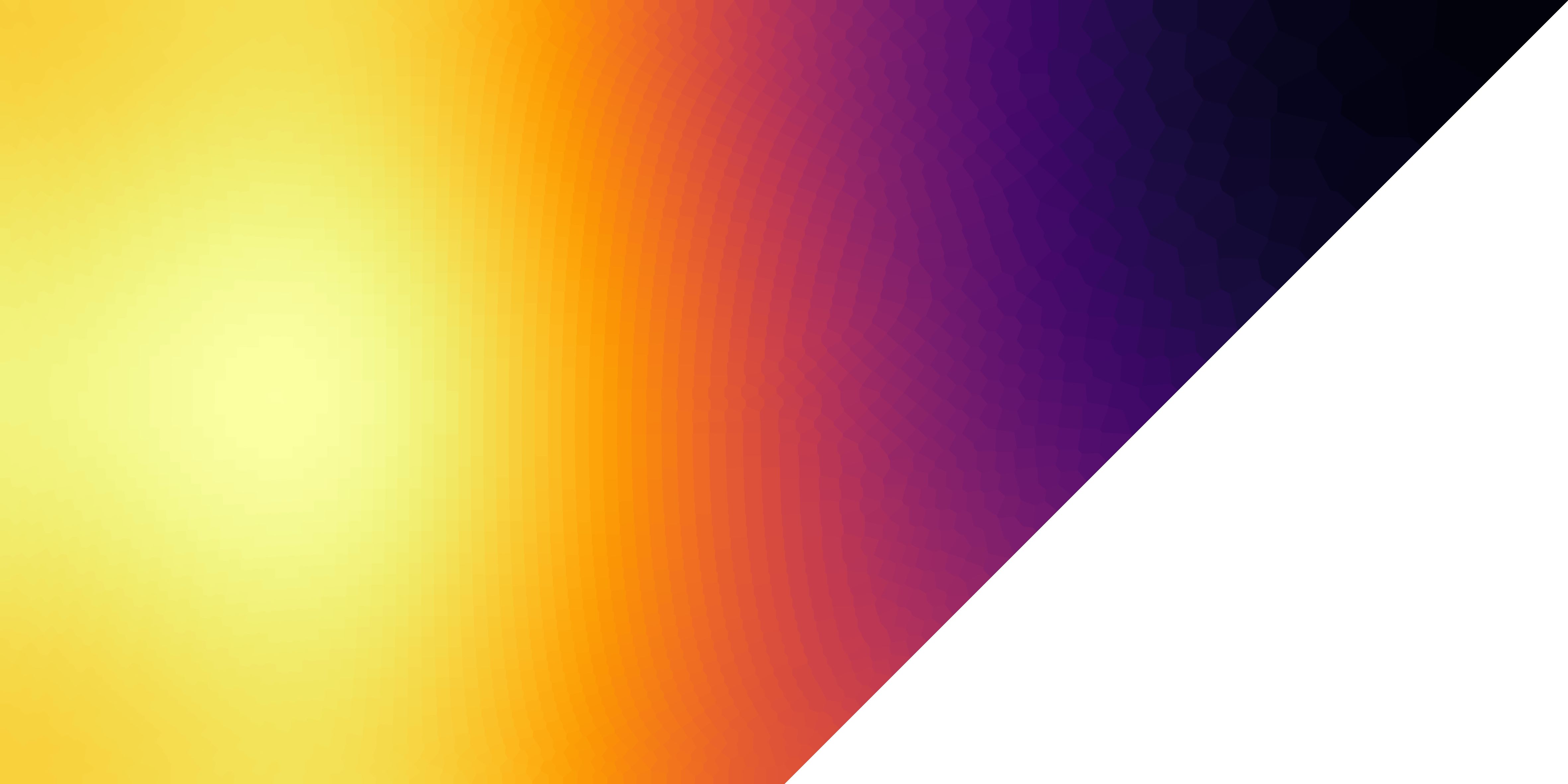}
   \includegraphics[width=.32\textwidth]{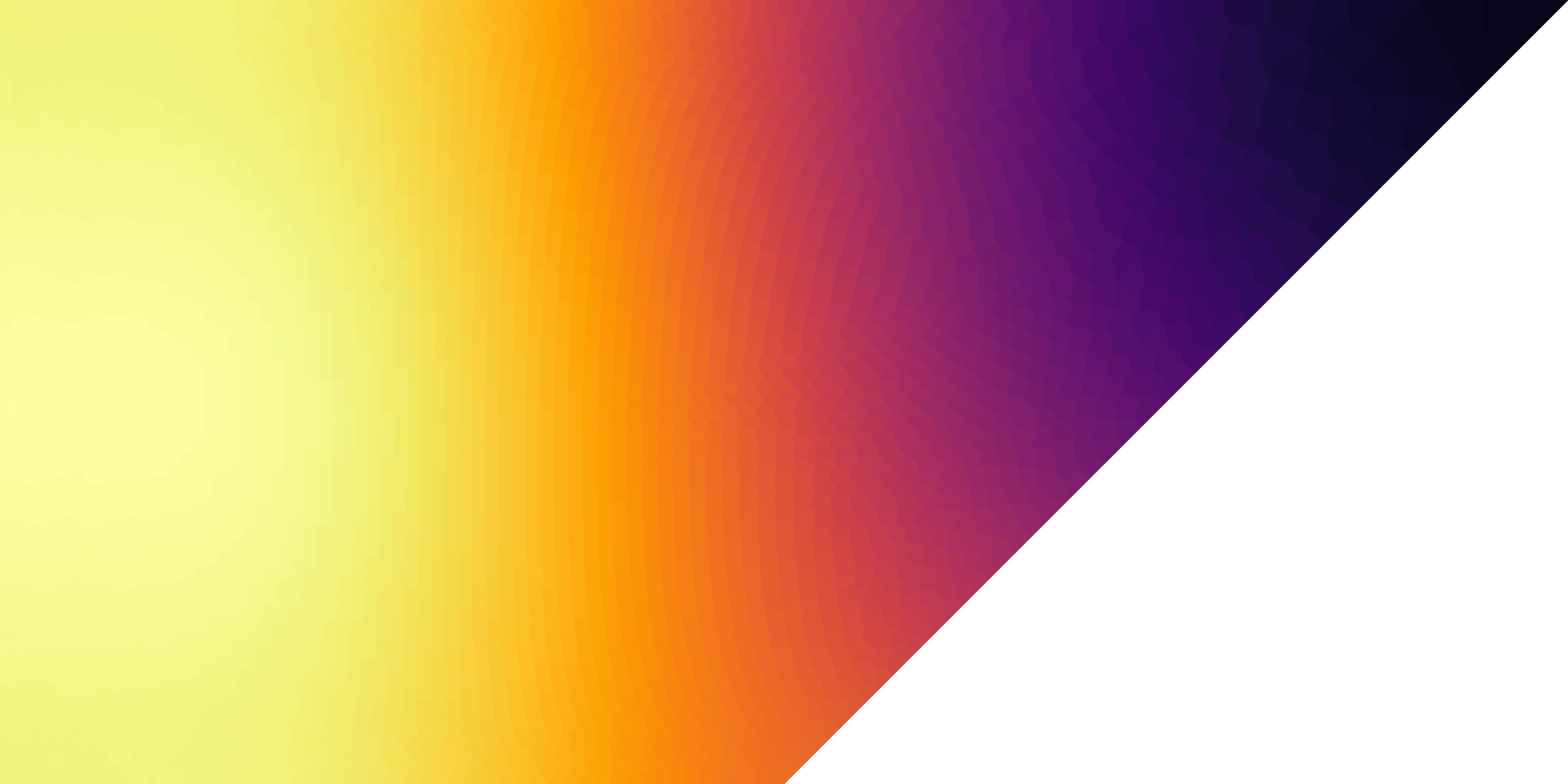}
      \end{center}
   \caption{Simulation of the heat equation, at several timesteps, for
     $h=\frac{1}{80}$. \label{fig:heat80}}
   \vspace{-.5cm}
 \end{figure}

  \begin{figure}
   \begin{center}
     \includegraphics[width=.4\textwidth]{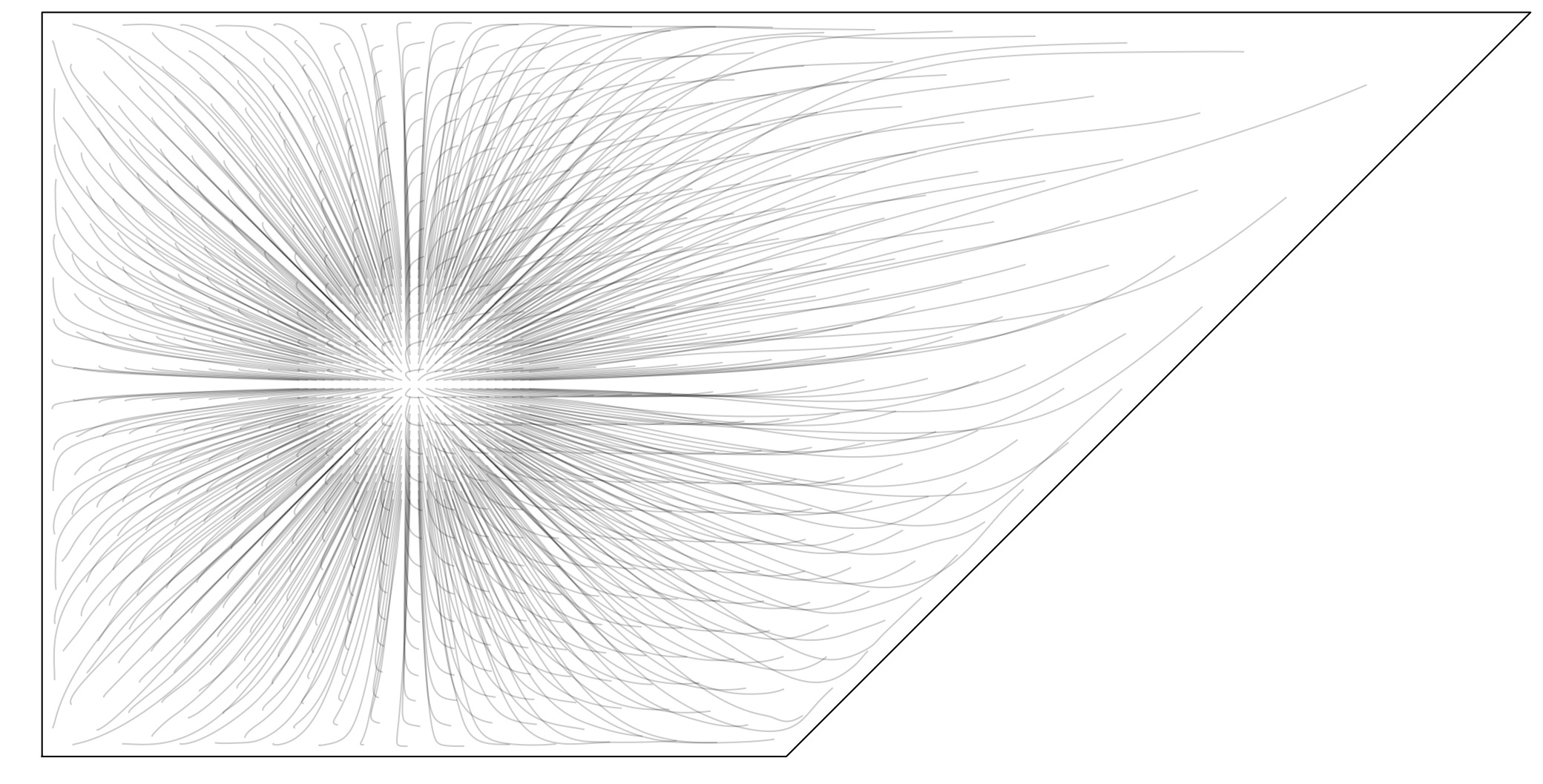}
     \includegraphics[width=.4\textwidth]{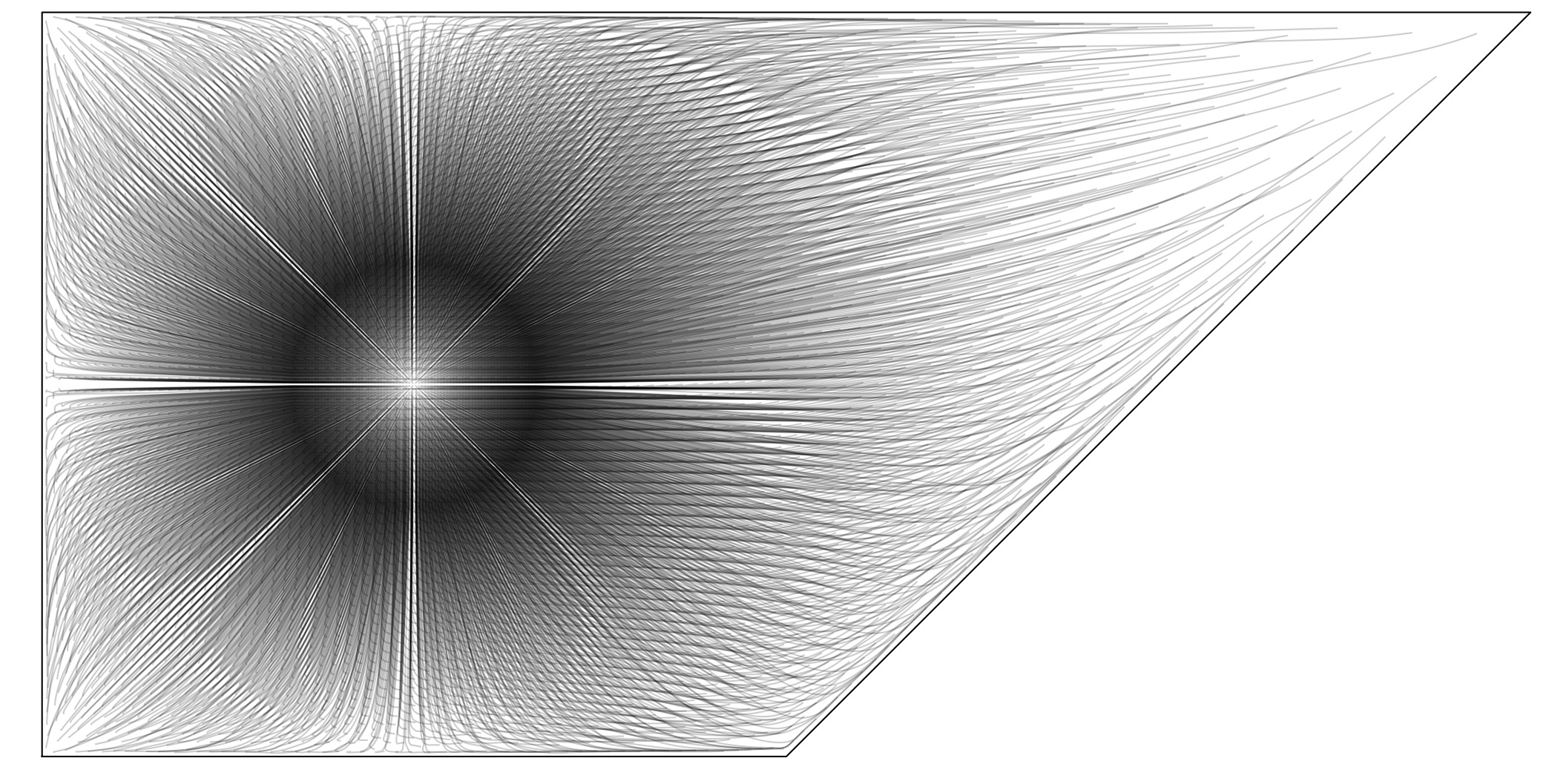}
   \end{center}
   \caption{Trajectories of particles along the heat flow (see Rem. \ref{rem:lagheat}). \emph{Left: $h=\frac{1}{30}$, Right: $h=\frac{1}{80}$}. \label{fig:heatlag}}
   \end{figure}